\documentclass[reqno,a4paper,12pt]{amsart}

\usepackage[T1]{fontenc}
\usepackage[english]{babel}
\usepackage[babel]{microtype}
\usepackage[a4paper, margin=1.1in]{geometry}

\usepackage{amsmath,amssymb,amsthm,amsrefs}

\usepackage{mathrsfs}

\usepackage{mathtools}
\usepackage{commath}

\usepackage{bbm}

\usepackage{cases}
\usepackage{enumitem}
\setlist[enumerate,1]{label=(\roman*)}

\usepackage{centernot}
\usepackage{booktabs}
\usepackage{multirow}
\usepackage{comment}
\usepackage{float}
\usepackage{scalerel}

\usepackage[pdftex, pdfborderstyle={/S/U/W 0}]{hyperref}
\hypersetup{
    colorlinks=true,
    linkcolor=magenta,
    citecolor=cyan,
}

\usepackage{cleveref}

\usepackage{etoolbox}

\usepackage{subfiles}
\usepackage{pdfpages}

\numberwithin{equation}{section}

\theoremstyle{plain}

\newtheorem{theorem}    {Theorem}[section]
\newtheorem{lemma}      [theorem] {Lemma}

\newtheorem{proposition}[theorem] {Proposition}
\newtheorem{corollary}  [theorem] {Corollary}

\theoremstyle{definition}
\newtheorem{definition} [theorem] {Definition}

\newtheorem{claim}      [theorem] {Claim}

\newtheorem{strategy}   [theorem] {Strategy}
\newtheorem{observation}[theorem] {Observation}

\newtheorem*{remark}    {Remark}

\newcommand{\indef}[1]{\emph{#1}}

\newcommand{\defined}{\mathrel{\coloneqq}}

\DeclarePairedDelimiter{\p}{\lparen}{\rparen}

\renewcommand{\le}{\leqslant}
\renewcommand{\leq}{\leqslant}
\renewcommand{\ge}{\geqslant}
\renewcommand{\geq}{\geqslant}

\let\oldexists\exists
\let\exists\relax
\DeclareMathOperator{\exists}{\oldexists}
\let\oldforall\forall
\let\forall\relax
\DeclareMathOperator{\forall}{\oldforall}

\newcommand{\st}{\mathbin{\colon}}
\undef{\set}
\DeclarePairedDelimiter{\set}{\lbrace}{\rbrace}

\DeclarePairedDelimiter{\card}{\lvert}{\rvert}

\DeclarePairedDelimiter{\floor}{\lfloor}{\rfloor}
\DeclarePairedDelimiter{\ceil}{\lceil}{\rceil}

\undef{\mod}
\newcommand{\mod}[1]{\ (\mathrm{mod}\ #1)}

\DeclareMathOperator{\boundary}{\partial}

\undef{\abs}
\DeclarePairedDelimiterX{\abs}[1]
  {\lvert}{\rvert}{\ifblank{#1}{\,\cdot\,}{#1}}
\undef{\norm}
\DeclarePairedDelimiterX{\norm}[1]
  {\lVert}{\rVert}{\ifblank{#1}{\,\cdot\,}{#1}}
\DeclarePairedDelimiterX{\inner}[2]
  {\langle}{\rangle}{\ifblank{#1}{\,\cdot\,}{#1},\ifblank{#2}{\,\cdot\,}{#2}}

\DeclareMathDelimiter{\given}
  {\mathbin}{symbols}{"6A}{largesymbols}{"0C}
\DeclareMathOperator{\Prob}{\mathbb{P}}
\DeclarePairedDelimiterXPP{\prob}[1]
  {\Prob}{\lparen}{\rparen}{}
  {\renewcommand{\given}{\nonscript\;\delimsize\vert\nonscript\;\mathopen{}}
  \ifblank{#1}{\,\cdot\,}{#1}}
\DeclareMathOperator{\Expec}{\mathbb{E}}
\DeclarePairedDelimiterXPP{\expec}[1]
  {\Expec}{\lparen}{\rparen}{}
  {\renewcommand{\given}{\nonscript\;\delimsize\vert\nonscript\;\mathopen{}}
  \ifblank{#1}{\,\cdot\,}{#1}}
\DeclareMathOperator{\Var}{Var}
\DeclarePairedDelimiterXPP{\var}[1]
  {\Var}{\lparen}{\rparen}{}
  {\renewcommand{\given}{\nonscript\;\delimsize\vert\nonscript\;\mathopen{}}
  \ifblank{#1}{\,\cdot\,}{#1}}
\DeclareMathOperator{\Cov}{Cov}
\DeclarePairedDelimiterXPP{\cov}[2]
  {\Cov}{\lparen}{\rparen}{}{#1,#2}

\newcommand{\NN}{\mathbb{N}}

\newcommand{\RR}{\mathbb{R}}

\newcommand{\ZZ}{\mathbb{Z}}

\newcommand{\cF}{\mathcal{F}}

\newcommand{\cR}{\mathcal{R}}

\usepackage{tikz}
\usetikzlibrary{decorations.pathreplacing}

\DeclareMathOperator{\bb}{bb}
\newcommand{\Pperc}[1]{\Prob_{#1}}
\DeclarePairedDelimiterXPP{\pperc}[2]
  {\Pperc{#1}}{\lparen}{\rparen}{}
  {\renewcommand{\given}{\nonscript\;\delimsize\vert\nonscript\;\mathopen{}}
  \ifblank{#2}{\,\cdot\,}{#2}}

\newcommand{\Porient}[1]{\Prob_{#1}^{\rightarrow}}
\DeclarePairedDelimiterXPP{\porient}[2]
  {\Porient{#1}}{\lparen}{\rparen}{}
  {\renewcommand{\given}{\nonscript\;\delimsize\vert\nonscript\;\mathopen{}}
  \ifblank{#2}{\,\cdot\,}{#2}}

\begin{document}

\title[The Maker-Breaker percolation game on the square lattice]
    {The Maker-Breaker percolation game \\ on the square lattice}

\author{Vojt\v{e}ch Dvo\v{r}\'ak \and Adva Mond \and Victor Souza}
\address{Department of Pure Mathematics and Mathematical Statistics (DPMMS), University of Cambridge, Wilberforce Road, Cambridge, CB3 0WA, United Kingdom}
\email{\{vd273,am2759,vss28\}@cam.ac.uk}

\begin{abstract}
We study the $(m,b)$ Maker-Breaker percolation game on $\mathbb{Z}^2$, introduced by Day and Falgas-Ravry.
As our first result, we show that Breaker has a winning strategy for the $(m,b)$-game whenever $b \geq (2-\frac{1}{14} + o(1))m$, breaking the ratio $2$ barrier proved by Day and Falgas-Ravry.

Addressing further questions of Day and Falgas-Ravry, we show that Breaker can win the $(m,2m)$-game even if he allows Maker to claim $c$ edges before the game starts, for any integer $c$, and that he can moreover win rather fast (as a function of $c$).

Finally, we consider the game played on $\mathbb{Z}^2$ after the usual bond percolation process with parameter $p$ was performed.
We show that when $p$ is not too much larger than $1/2$, Breaker almost surely has a winning strategy for the $(1,1)$-game, even if Maker is allowed to choose the origin after the board is determined.
\end{abstract}

\maketitle


\section{Introduction}
\label{sec:intro}

\subsection{Background}

The so-called \emph{Maker-Breaker games} are a large and well studied class of positional games.
To define the simplest version of a Maker-Breaker game, we need a finite or infinite set $\Lambda$ and a family $\cF$ of subsets of $\Lambda$.
The set $\Lambda$ is called the \emph{board} of the game, and $\cF$ is the collection
of \emph{winning sets} or \emph{winning configurations}.
The game is played in \emph{rounds}.
In each round, Maker and Breaker respectively claim an as yet unclaimed element of $\Lambda$, where Maker is the first player.
Breaker wins the game if he claims at least one element in each $F \in \cF$ by any finite point of the game.
Otherwise, Maker wins.
On a finite board, this is equivalent to Maker claiming all elements of some $F \in \cF$ by the end of the game, though the same is not true for infinite boards.
This version of the game is also called the \emph{unbiased Maker-Breaker game}.
In the \emph{biased Maker-Breaker game}, introduced by Chvátal and Erd\H{o}s \cite{chvatal1978biased}, the players may claim more elements.
To be precise, given natural numbers $m$ and $b$, in each round of the $(m, b)$
Maker-Breaker game, Maker claims $m$ elements of the board whereas Breaker claims $b$.
For more information about Maker-Breaker games as well as other positional games, see the
excellent books of Beck \cite{beck2008combinatorial}, and of Hefetz, Krivelevich, Stojaković and Szabó~\cite{hefetz2014positional}.

In this paper, we address the following Maker-Breaker game played on an infinite board, introduced by Day and Falgas-Ravry~\cites{day2020maker1,day2020maker2}.
Let $\Lambda$ be an infinite connected graph and let $v_0 \in V(\Lambda)$ be a vertex.
In the scope of this paper, we only have $\Lambda$ being $\ZZ^2$ or an infinite connected subgraph of it.
The \emph{$(m,b)$ Maker-Breaker percolation game on $(\Lambda, v_0)$} is the game with board $E(\Lambda)$ where the winning sets are all infinite connected subgraphs of $\Lambda$ containing $v_0$.
That is,  Maker's goal is to ensure that $v_0$ is always contained in an infinite subgraph of $\Lambda$ spanned by the edges that she claimed and the unclaimed edges.
Note that Breaker wins the game by claiming at least one element in each winning set.
In this game this means that Breaker's goal is to claim any set of edges separating $v_0$ from infinity.
Notably, the $(1,1)$-game on $\Lambda$ can be seen as a generalisation
of the well-known Shannon switching game to an infinite board,
see Lehmann~\cite{lehman1964solution} for a description and a solution of this game.

Throughout the paper, we refer to the $(m,b)$ Maker-Breaker percolation game on $(\Lambda, v_0)$ as $(m,b)$-game on $(\Lambda, v_0)$.
If $\Lambda$ is a transitive graph, we omit the vertex $v_0$ in our notation, as it does not change the analysis of the game.

Next, we summarise the main results of the paper \cite{day2020maker2} of Day and Falgas-Ravry.

\begin{theorem}
\label{thm:original}
  Let $m, b \in \NN$. Then
  \begin{enumerate}
    \item Maker has a winning strategy for the $(1,1)$-game on $\ZZ^2$;
    \item If $m \geq 2b$, then Maker has a winning strategy
            for the $(m,b)$-game on $\ZZ^2$;
    \item If $b \geq 2m$, then Breaker has a winning strategy
            for the $(m,b)$-game on $\ZZ^2$.
  \end{enumerate}
\end{theorem}

Note that clearly, neither player is harmed by having more moves on their turn, so if for instance, Breaker wins the $(m,b)$-game on a board, he also wins the $(m,b')$-game on that same board with $b' \geq b$.
This property is called \emph{bias monotonicity}.

In view of \Cref{thm:original},
Day and Falgas-Ravry raised many interesting questions.
Most strikingly, they asked if there is some critical ratio $\rho^*$ such that, there exists a positive function $\varphi(m) = o(m)$ such that Breaker wins the $\p[\big]{m, \rho^*m  + \varphi(m)}$-game and Maker wins the $\p[\big]{m, \rho^*m - \varphi(m)}$-game.
\Cref{thm:original} shows that if such ratio exists, then $1/2 \leq \rho^* \leq 2$.
These bounds are associated with the fact that a set of $k$ connected edges in $\ZZ^2$ has edge-boundary of size at most $2k + 4$, see \Cref{lem:perimetric}.
By the \emph{perimetric barrier} we refer to the limitation of either player being roughly twice as powerful as the the other player.
Although the main problem they suggest is to break the perimetric barrier, they set out as an open problem to show whether Breaker or Maker can win the $(m,2m-1)$ or $(2b-1,b)$ games, respectively.

\subsection{Our results}

As our first result, we break the perimetric barrier on Breaker's side, making progress towards answering \cite{day2020maker2}*{Question 5.5} of Day and Falgas-Ravry about the critical ratio.

\begin{theorem}
\label{thm:main1}
Consider the $(m,b)$ Maker-Breaker percolation game on $\ZZ^2$, where $m \ge 29$ and $b\ge 2m - s$ for some $0 \le s \le \frac{m-22}{14}$.
Then Breaker has a winning strategy, which moreover ensures that he wins within the first $3$ rounds of the game.
\end{theorem}

In particular, this shows that if $\rho^*$ exists,
then $1/2 \leq \rho^* \leq 27/14 \approx 1.93$,
and thus, breaks the perimetric barrier as discussed previously.
We do not believe this bound to be tight,
and moreover, we did not attempt to optimise for this constant,
as we also believe that the current method will not yield
the optimal bound.

\Cref{thm:main1} improves the ratio on Breaker's side for $m \ge 36$, and it also shows that for $m \ge 29$, Breaker wins the $(m,2m)$-game rather fast.
Nonetheless, it is also of interest to determine how powerful Breaker is in the $(m,2m)$-game for smaller values of $m$.
In the proof of \Cref{thm:original}, Day and Falgas-Ravry
show that Breaker can win the $(m,2m)$-game on $\ZZ^2$ within $m^{16m+O(1)}$ rounds, and ask~\cite{day2020maker2}*{Question 5.7} how far this is from best possible.
Concerning this question, we prove a slightly stronger result, showing that Breaker can win fast even when allowing Maker an initial boost in the form of an option to claim some edges before the game starts.

We consider the following variant of the game.
For integers $m,b \ge 1$ and $c \ge 0$ define the \emph{$c$-boosted $(m,b)$ Maker-Breaker percolation game} on $\ZZ^2$ to be the same as the $(m,b)$ percolation game, with the addition that only in her very first turn, Maker claims $c$ extra edges (so overall in her first round she claims $m+c$ edges).
Concerning this game, Day and Falgas-Ravry asked~\cite{day2020maker2}*{Question 5.6} whether Breaker having
a winning strategy for the $(m,b)$-game on $(\Lambda, v_0)$
implies that he also has a winning strategy for the $c$-boosted
version of the same game.

In view of~\cite{day2020maker2}*{Questions 5.6, 5.7}, we prove the following result.

\begin{theorem}
\label{thm:boost}
Let $m \ge 1$ and $c \ge 0$ be integers, and let $b \ge 2m$.
Then Breaker wins the $c$-boosted $(m,b)$ Maker-Breaker percolation game on $\ZZ^2$, and moreover, he can ensure to win within the first $(2c+4)(2c+5)\p[\big]{\ceil{ \frac{2c+2}{m}} + 1}$ rounds.
\end{theorem}

This theorem tells us that Breaker can not only win the $(m,2m)$-game on $\ZZ^2$ quite fast, and can not only win the $c$-boosted game for any $c$, he can also win quite fast the $c$-boosted game.
Moreover, the number of rounds Breaker needs is uniformly bounded in $m$ and polynomial in $c$.
Thus, for the $(m,2m)$-game, we answer the stronger combined version of Questions 5.6 and 5.7 of \cite{day2020maker2}.

Applying \Cref{thm:boost} with $c = 0$, we get the following extension of \Cref{thm:main1} for the $(m,2m)$-game without an initial boost.

\begin{corollary}
\label{cor:m2mBreaker}
Let $m \geq 1$ and $b \ge 2m$ be integers.
Then Breaker can guarantee to win the $(m,b)$ percolation game on $\ZZ^2$ within the first $40$ rounds of the game.
Moreover, if $m\ge 29$ then Breaker wins within $3$ rounds.
\end{corollary}

Note that this cannot be extended for a win in $3$ rounds for every $m$, as in fact, for $m = 1$, Maker can survive for $5$ rounds.
In the range $1 \leq m \leq 28$, the bound of $40$ we obtain is not optimal, as the proof of \Cref{thm:boost} specialised to $c = 0$ could be greatly simplified.
For $m \geq 29$, Maker can indeed survive for $3$ rounds, as it becomes clear in the proof of \Cref{thm:main1}.

The final subject we address regards the Maker-Breaker percolation game being played on a random board.
It is very common to study Maker-Breaker games on random boards in general.
In this case, however, it is of special interest since it was originally introduced by an analogy to percolation theory.

It is then very natural to consider the variant of the Maker-Breaker percolation game on $(\ZZ^2)_p$, the subgraph of open edges in the usual bond percolation process on $\ZZ^2$ with parameter $p \in [0,1]$.
It is well known that for $p \leq 1/2$, there are no infinite connected components on $(\ZZ^2)_p$, and if $p \in (1/2,1)$, the probability that the origin is in an infinite connected component is strictly between zero and one.
To avoid an easy win of Breaker in round zero, we grant Maker the power to choose the origin $v_0$
after sampling the configuration.
We call this the \emph{Maker-Breaker percolation game on $p$-polluted $\ZZ^2$}.

Day and Falgas-Ravry asked what is a value of the critical threshold $\theta^{*} \in [1/2,1]$ such that when $p < \theta^{*}$, Breaker almost surely has a winning strategy for the $(1,1)$-game on $p$-polluted $\ZZ^2$.
We make a first progress in this direction, providing a non-trivial lower bound for this threshold.
To do so, we explore a connection to a different percolation model, the north-east oriented percolation model on $\ZZ^2$, which is known to have a critical threshold $p^*$ (see \Cref{sec:polluted} for more details about this model and about the percolation game on a $p$-polluted board).

\begin{theorem}
\label{thm:polluted}
  Let $p^{*}$ be a critical threshold for the north-east oriented percolation on $\ZZ^2$.
  Then for any $p < p^{*}$, Breaker almost surely has a winning strategy for the $(1,1)$-game on $p$-polluted $\ZZ^2$.
\end{theorem}

In particular, as $p^{*} > 0.6298$ by a result of Dhar~\cite{dhar1982directed}, for any $p \le 0.6298$, Breaker almost surely has a winning strategy for $(1,1)$-game on $(\ZZ^2)_p$.
Although considering the critical threshold for the north-east oriented percolation yields a non-trivial result for Breaker's win on the polluted board, this strategy cannot be extended to every $p < 1$, since it was shown by Balister, Bollobás and Stacey~\cite{balister1994improved} that $p^* < 0.6863$.

\subsection{Tools and strategy}
\label{subsec:tools}

Throughout the paper we use two important tools.
The first is relating our game to an auxiliary game, where Maker either has to keep her graph connected, or at least connected in some generalised sense.
However, if we want to claim that it is enough to prove that Breaker wins against a restricted Maker, we have a certain price to pay.
In particular, we consider only strategies of Breaker in which he claims edges from the edge-boundary of Maker's connected component, or a slightly generalised version of that.
More importantly, we enable Maker to `save' some edges for later, to make the auxiliary game resemble the original one.
Despite these changes, this setting ends up being much easier to analyse, as one of the hard things to tackle when considering strategies for Breaker is handling different connected components in Maker's graph.
Furthermore, it turns out that with these adjustments, analysing the auxiliary game is indeed sufficient to prove our results for the original game.

Our second tool is considering variations on \Cref{lem:perimetric} (Lemma 2.3 in \cite{day2020maker1}).
This simple result tells us that the edge-boundary of any connected finite subgraph of $\ZZ^2$ is at most `a bit' larger than twice the number of edges of this subgraph.
Hence, for instance, when playing the $(m,2m)$-game, it is enough to force Maker to play several `bad moves', not enlarging the edge-boundary of her graph by too much, so that Breaker can surround her connected component completely.

When we use a more general notion of connectivity, we need a slightly more general version of \Cref{lem:perimetric}.
This variant allows us to analyse the game in a global sense.
This, in particular, is how we manage to break the perimetric barrier.

Finally, to prove \Cref{thm:polluted}, we only use some simpler tools.
In particular, one of those is a pairing strategy in the same spirit of the one that Day and Falgas-Ravry~\cite{day2020maker2} used in their proof for Maker's win in the $(1,1)$-game, but this time as a strategy for Breaker.

\subsection{Organisation}

Firstly, we present our notation and the definitions we work with in \Cref{sec:preliminary}.
After that, we prove \Cref{thm:main1} in \Cref{sec:ratio} and \Cref{thm:boost} in \Cref{sec:fast-boost}. We discuss the game played on the polluted board and prove \Cref{thm:polluted} in \Cref{sec:polluted}. Finally, we present several open problems and further directions in \Cref{sec:open}.


\section{Preliminaries}
\label{sec:preliminary}

For a graph $G$, we denote by $V(G)$ its vertex set, and by $E(G)$ its edge set. Furthermore, we set $e(G) \defined \card{E(G)}$.
For a subset of vertices $U \subseteq V(G)$ we denote by $G[U]$ the subgraph of $G$ induced by $U$.

The \emph{edge boundary} $\boundary H$ of a finite subgraph $H$ of a possibly infinite graph $G$ is
\begin{equation*}
    \boundary H \defined \set[\big]{ \set{x,y} \in E(G) \setminus E(H) \st  \set{x,y} \cap V(H) \neq \emptyset }.
\end{equation*}
We usually abbreviate `edge boundary' to `boundary', as we do not consider any other type of boundary in this paper.

We use the standard terminology where by the square lattice $\ZZ^2$, we mean the infinite graph with the following vertex and edge sets.
\begin{align*}
  V(\ZZ^2) &\defined \set[\big]{ (x,y) \st x,y \in \ZZ }, \\
  E(\ZZ^2) &\defined \set[\big]{ \set{(x,y), (x',y')} \subseteq\ZZ^2 \st
    \abs{x - x'} + \abs{y - y'} = 1}.
\end{align*}

When considering probability in \Cref{sec:polluted}, we say that an event occurs almost surely (abbreviated a.s.) if it occurs with probability $1$.

\subsection{Useful lemmas}
Throughout the paper, we use several times the following reverse isoperimetric inequality observed by Day and Falgas-Ravry~\cite{day2020maker1}.

\begin{lemma}
\label{lem:perimetric}
  Let $C$ be a finite connected subgraph of $\ZZ^2$, then
  \begin{equation*}
    \card{\boundary C} \leq 2e(C) + 4.
  \end{equation*}
\end{lemma}

We present two more versions of this lemma, for which we need some definitions.

\begin{definition}
\label{def:box}
    We say that a finite subgraph $B\subseteq \ZZ^2$ is a \emph{box}, if it is induced by a set of vertices of the form
\begin{align*}
    \set[\big]{ (x,y) \st a \leq x \leq b, \, c \leq y \leq d },
\end{align*}
    for some $a,b,c,d \in \ZZ$.
\end{definition}

\begin{definition}\label{def:bounding-box}
    Let $S \subseteq \ZZ^2$ be a finite set of edges and let $V_S$ be the set vertices in the graph it spans.
    The \emph{bounding box} of $S$, denoted $\bb(S)$, is the minimal box in $\ZZ^2$ containing $S$.
    To spell it out, let
    \begin{align*}
        m_x(S) \defined \min \set[\big]{ x \st (x,y) \in V_S },
        && M_x(S) \defined \max\set[\big]{ x \st (x,y) \in V_S }.
    \end{align*}
    Also, let $m_y(S), M_y(S)$ be defined analogously for the $y$-axis.
    Then $\bb(S)$ is the box induced by the following set of vertices
    \begin{equation*}
        \set[\big]{ (x,y) \st m_x(S) \le x \le M_x(S), \, m_y(S) \le y \le M_y(S) }.
    \end{equation*}
\end{definition}

For a finite subgraph $G$ of $\mathbb{Z}^2$, we write $\bb(G)$ for $\bb(E(G))$.

\begin{figure}
\centering
\begin{tikzpicture}[scale=0.5]
    \tikzstyle{s-helper}=[lightgray, thin];
    \definecolor{s-b-col}{RGB}{215,25,28};
    \tikzstyle{s-b}=[s-b-col, ultra thick];
    \definecolor{s-o-col}{RGB}{253,174,97};
    \tikzstyle{s-o}=[s-o-col, ultra thick];
    \definecolor{s-m-col}{RGB}{171,217,233};
    \tikzstyle{s-m}=[s-m-col, ultra thick];
    \definecolor{s-g-col}{RGB}{220,220,220};
    \tikzstyle{s-g}=[s-g-col, line width=0.1pt];
    \clip (-7.5,-2) rectangle (5.9,7.6);
    \draw[s-m]
        (-7,0.5) to[out=20,in=150] (-3.5,4.2)
        (-7.3,3.6) to[out=35, in=-150] (-1.1,1.1);
    \draw[s-m]
        (-2.1, 7) to[out=45, in=60] (1.7, 3.8)
        (-0.5, 5.7) to[out=75, in=-150] (3.2, 6.7);
    \draw[s-m]
        (1.8, -1.7) to[out=-15, in=160] (5.7, 1.6)
        (3.5, -0.5) to[out=-30, in=-150] (4.8, -1.2);
    \draw[darkgray, densely dashed, thick] (-7.4,0.4) rectangle (-1,4.5);
    \draw[darkgray, densely dashed, thick] (-2.2,3.7) rectangle (3.3,7.5);
    \draw[darkgray, densely dashed, thick] (1.7,-1.9) rectangle (5.8,1.8);
    \draw[black, thick] (-7.4,-1.9) rectangle (5.8,7.5);
\end{tikzpicture}
\caption{Three connected components, their bounding boxes (dashed),
and the bounding box of their box-component (solid).}
\label{fig:bounding-box}
\end{figure}
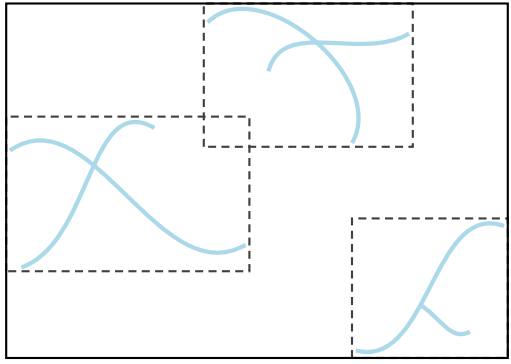

\begin{lemma}\label{lem:perim-bb}
    Let $D$ be a finite connected subgraph of $\ZZ^2$, then
  \begin{equation*}
    \card{\boundary \bb(D)} \le \card{\boundary D}.
  \end{equation*}
\end{lemma}
\begin{proof}
    Let $\boundary D$ consist of $h$ horizontal and $v$ vertical edges, so that $\card{\boundary D}$.
    Let $m_x = m_x(D)$, $M_x = M_x(D)$, $m_y = m_y(D)$, and $M_y = M_y(D)$ be as in \Cref{def:bounding-box}.
    Note that for any $m_x \le x_0 \le M_x$ there are at least two vertical boundary edges in $\boundary D$ of the form $\set{ (x_0, y), (x_0, y+1)}$.
    Analogously, for any $m_y \le y_0 \le M_y$ there are at least two horizontal boundary edges in $\boundary D$ of the form $\set{ (x, y_0), (x+1, y_0)}$.
    In particular we have $h \ge 2(M_y - m_y + 1)$ and $v \ge (M_x - m_x + 1)$.
    Recall that the box $\bb(D)$ has sides of lengths $M_x - m_x$ and $M_y - m_y$, so
    \begin{equation*}
        \card{\boundary \bb(D)} = 2(M_x - m_x + M_y - m_y) + 4,
    \end{equation*}
    and the result follows.
\end{proof}

We define a generalisation of connected component that incorporates the notion of bounding box in \Cref{def:bounding-box}.
For doing so we go through several definitions.

\begin{definition}
\label{def:box-intesect}
    Let $D_1, D_2 \subseteq \ZZ^2$ be boxes.
    We say that $D_1, D_2$ \emph{box-intersect} if
    \begin{equation*}
        \p[\big]{E(D_1) \cup (\boundary D_1)} \cap E(D_2) \neq \emptyset.
    \end{equation*}
\end{definition}

Note that the relation in \Cref{def:box-intesect} is symmetric.

\begin{definition}\label{def:box-comp}
    Let $D \subseteq \ZZ^2$ be a finite subgraph.
    Let $C_1, \ldots C_t$ be the connected components of $D$, and let $\cR \defined \set{\bb(C_1), \ldots, \bb(C_t)}$ be the collection of their bounding boxes.
    As long as possible, repeat the following process.
    If there exist $R_i, R_j \in \cR$ which box-intersect, remove them from $\cR$ and replace them by the bounding box of their union, that is, by $\bb(R_i \cup R_j)$.
    The final $\cR$ obtained in the end of this process is called a collection of \emph{box-components of $D$}.
    If the final $\cR$ contains precisely one box-component, then we say that $D$ is \emph{box-connected}.
\end{definition}

\begin{observation}
    Given a finite subgraph $D \subseteq \ZZ^2$, the output of the process defined in \Cref{def:box-comp} on $D$ is unique.
    Therefore, it induces an equivalence relation on the connected components of $D$.
    Note further that a box-component contains no isolated vertices, unless the component itself is a single vertex.
\end{observation}

Hence, for each subgraph $D\subseteq \ZZ^2$ we can consider its collection of box-components.
This allows us to state a slightly generalised version of \Cref{lem:perimetric}.

\begin{lemma}\label{lem:perim-box-comp}
    Let $D$ be a finite box-connected subgraph of $\ZZ^2$, then
    \begin{equation*}
        \card{\boundary \bb(D)} \le 2e(D) + 4.
    \end{equation*}
\end{lemma}
\begin{proof}
    If $D$ is connected, then the result follows immediately from combining \Cref{lem:perim-bb} and \Cref{lem:perimetric}.

    Otherwise, it is enough to prove the statement for the case where is $D$ a union of two graphs, $C_1, C_2$, for which the statement holds for both, and where $\bb(C_1), \bb(C_2)$ box-intersect.
    Indeed, we then apply this repeatedly for each step in the process defined in \Cref{def:box-comp}, which ends in a single box-component for $D$.
    Denote $R_i \defined \bb (C_i)$ for $i=1,2$, so by assumption we have
    \begin{equation}
    \label{eq:assump}
        \card{\boundary R_i} \le 2e(C_i) + 4.
    \end{equation}
    As $R_1, R_2$ box-intersect, we have that
    \begin{align*}
        \p[\big]{E(R_1) \cup (\boundary R_1)} \cap E(R_2) \neq \emptyset.
,    \end{align*}
    If either $R_1 \subseteq R_2$ or $R_2 \subseteq R_1$ then we either have
    \begin{equation*}
        \card{\boundary \bb(D)} = \card{\boundary R_1} \le 2e(C_1) + 4 < 2e(D) + 4,
    \end{equation*}
    or a similar relation holds with $R_1$ replaced by $R_2$.

    Otherwise, we have $\card{(\boundary R_1) \cap E(R_2)} \ge 1$ and $R_2 \nsubseteq R_1$.
    It follows by simple geometric observations that $\card[\big]{ \p[\big]{E(R_1) \cup (\boundary R_1)} \cap (\boundary R_2)} \ge 3$.
    In particular,
    \begin{equation*}
        \card{\boundary(R_1 \cup R_2)} \le \card{\boundary R_1} + \card{\boundary R_2} - 4.
    \end{equation*}
    Consequently,
    \begin{align*}
        \card{\boundary \bb(D)} &= \card{ \boundary \bb(C_1 \cup C_2)} \\
        &= \card{\boundary \bb(R_1 \cup R_2)} \\
        &\le \card{\boundary(R_1 \cup R_2)} && \p{\text{By \Cref{lem:perim-bb}}} \\
        &\le \card{\boundary R_1} + \card{\boundary R_2} - 4 \\
        &\le 2e(C_1) + 2e(C_2) + 4 && \p{\text{By (\ref{eq:assump})}} \nonumber \\
        &= 2e(D) + 4. \qedhere
    \end{align*}
\end{proof}

\begin{remark}
\label{rem:BoxBdryRec}
    Note that the edge boundary $\boundary B$ of any box $B \subseteq \ZZ^2$, when regarded as a set of dual edges, forms a rectangle.
\end{remark}


\section{Breaking the perimetric ratio}
\label{sec:ratio}

In this section, we prove \Cref{thm:main1}.
First, let us recall that by bias monotonicity, it is enough to prove \Cref{thm:main1} for $b=2m-s$, as having more power cannot harm Breaker.
As a first step, we describe an auxiliary game for which we show that a win of Breaker in this game implies a win of him in the original game.
After that, we provide Breaker with an explicit strategy.
The rest of this section is devoted to showing that Breaker, by following the suggested strategy, wins the auxiliary game within three rounds.
We prove this by a geometric analysis, relying crucially on the introduced notion of box-connectivity and our tools from \Cref{sec:preliminary}.

If Breaker plays only on the boundary, it is natural to arrive at the perimetric barrier of the ratio $2$, because of \Cref{lem:perimetric}.
More precisely, when Breaker only claims edges from the boundary of Maker's graph, he cannot react to her future moves in advance.
That is, in each turn, Maker is able to create as many new unclaimed boundary edges as possible, to which Breaker must respond.
To get around this, it is helpful for Breaker to consider the global structure of Maker's graph.
Indeed, in general terms, one could interpret our strategy as Breaker forcing Maker to claim edges in an already played region of the board.
This extra power from previous turns will lead to the improvement on the ratio.

When analysing the game from the point of view of Breaker, we wish to consider the graph of Maker as being always box-connected.
Hence, we define the auxiliary game where we consider the box-component of Maker's graph containing the origin as her graph in each round.
Consequently, we allow more flexibility in the number of edges that she can claim in each turn.
Also, when defining Breaker's strategy later, we must insist that he can only play in a certain way for a result about auxiliary game to translate into the result about the original game.

\begin{definition}[$(m,b)$ Maker-Breaker box-limited percolation game on $\ZZ^2$]
\label{def:box-lntd-game}
    Two players, Maker and Breaker alternate claiming
    yet unclaimed edges of a board $\ZZ^2$,
    starting in round 1 with Maker going first.
    \begin{itemize}
        \item In round $i$, Maker chooses a non-negative integer $m_i$ such that for every $i$,
            \begin{equation}
            \label{eq:misum}
                \sum_{j=1}^i m_j \leq im,
            \end{equation}
        and then claims $m_i$ unclaimed edges from $E(\ZZ^2)$.
        Moreover, Maker must play in a way that in the end of each of her turns, her edges must be in the box-component of $v_0$ (see \Cref{def:box-comp}).
        \item In each round, Breaker claims at most $b$ unclaimed edges.
        \item Breaker wins if the connected component of $v_0$ in the graph formed by Maker's edges and all unclaimed edges becomes finite.
        If Maker can ensure that this never happens, then she wins.
    \end{itemize}
\end{definition}

A key result for us is the following proposition relating the two games.

\begin{proposition}
\label{prop:boxlmtd-unlmted}
    Let $m, b \geq 1$ be integers.
    Assume that Breaker can ensure his win in the $(m,b)$ box-limited percolation game on $\ZZ^2$ within the first $k$ rounds by claiming only edges from the boundary of the bounding box of Maker's graph, or from inside the box itself.
    Then he can also ensure his win in the $(m,b)$ percolation game on $\ZZ^2$ within the first $k$ rounds.
\end{proposition}
\begin{proof}
We show that if Maker has a strategy to ensure that Breaker will not win within the first $k$ rounds of the $(m,b)$ percolation game,
then she can also ensure that Breaker will not win
within the first $k$ rounds of the box-limited percolation game,
assuming that Breaker claims only edges from the boundary of the bounding box of her graph, or from inside it.

Assume that Maker has such a strategy for the (unlimited) percolation game.
Then she can win the box-limited game by playing as follows.
Denote by $M$ the box-component spanned by the edges claimed by Maker.
Maker follows her winning strategy for the percolation game, and whenever this includes playing some edge $e$ that after the end of her turn would be in a box-component which is not $M$, she only marks this edge as an \emph{imaginary} edge and does not play it in that round.
However, she claims an imaginary edge $e$ right after the first time she plays some edge that puts $e$ in $M$.

Firstly, Maker can afford saving imaginary edges to claim later in the game, as the terms $m_j$ only have to satisfy $\sum_{i=1}^i m_j \leq im$ for any $i \geq 1$.
Furthermore, Breaker cannot claim an imaginary edge before Maker claims it, as we assume that Breaker wins by claiming only edges in $(\boundary M) \cup E(M)$.
The result follows.
\end{proof}

Now consider the \emph{$(m, 2m - s)$ Maker-Breaker box-limited percolation game} on $\ZZ^2$, where $m \geq 36$ and $1 \leq s \leq \frac{m - 22}{14}$.

We provide Breaker with the following strategy.

\begin{strategy}[Breaker's strategy for the $(m,b)$ box-limited percolation game on $\ZZ^2$]
\label{str:B-box-lmtd}
For any $i \geq 1$, let $M_i$ be the set of edges claimed by Maker in her $i$-th turn.
Breaker plays according to the following steps.
If at any point of the game Breaker cannot follow any particular step, he forfeits the game.
\begin{description}[labelindent=\parindent, leftmargin=2\parindent]
    \item[First round] \hfill\\
    Set $B_1 \defined \bb(M_1)$.
    \begin{enumerate}
        \item[(1)]
        If $\card{\boundary B_1} \leq 2m-s$, claim all edges in $\boundary B_1$.

        \item[(2)]
        Otherwise, let $g_1 \defined \card{\boundary B_1}- 2m+s$.
        Claim $2m - s$ edges from $\boundary B_1$, leaving $g_1$ unclaimed boundary-edges in the middle (up to being possibly shifted by one edge, for parity reasons) of one of the longer sides of the box $B_1$.
        Denote by $G_1$ this set of $g_1$ unclaimed edges.
    \end{enumerate}

    \item[Second round] \hfill
    \begin{enumerate}
        \item[(1)]
        If $M_2 \cap G_1 = \emptyset$, then claim all edges in $G_1$ if possible, or forfeit if not possible.

        \item[(2)]
        Otherwise, $M_2 \cap G_1 \neq \emptyset$.
        Let $V_1$ be the set of vertices in $\ZZ^2 \setminus B_1$ which are contained in edges of $G_1$.
        Let $P_1 \defined E \p*{ \ZZ^2[V_1] }$ be the set of edges in the path induced by the vertices $V_1$.
        Let $C_1 \defined E(B_1) \cup \boundary B_1$
        and $B_2 \defined \bb((M_2 \cup P_1) \setminus C_1)$.

        \item[(2.1)]
        If $\card{(\boundary B_2) \setminus C_1} \le 2m-s$, claim all edges in $(\boundary B_2) \setminus C_1$.

        \item[(2.2)]
        Otherwise, let $g_2 \defined \card{(\boundary B_2) \setminus C_1} - 2m + s$.
        As $G_1$ is a set of boundary-edges in the middle of one of the longer sides of $B_1$, it splits the boundary edges adjacent to this side into two sets of consecutive boundary-edges.
        Denote these two sets by $L_1$ and $R_1$, such that $\card{R_1 \cap (\boundary B_2)} \leq \card{L_1 \cap (\boundary B_2 )}$.
        Let $e$ be an edge in $(\boundary B_2) \setminus C_1$ of minimal distance to $G_1$.
        Let $G_2$ be $g_2$ consecutive edges in $(\boundary B_2) \setminus C_1$, starting from $e$ (see \Cref{fig:round2} for an illustration).
        Claim all $2m-s$ edges in $(\boundary B_2) \setminus C_1$ excluding those edges in $G_2$.
    \end{enumerate}

    \item[Third round] \hfill
    \begin{enumerate}
        \item If there is a set of at most $2m-s$ unclaimed edges such that claiming them ensures a win in this round, claim all edges in this set.
        \item Otherwise, forfeit.
    \end{enumerate}
\end{description}
\end{strategy}

While the description of the strategy may seem complicated at first, it is in fact very simple.
For illustrations of the geometric content of this strategy, see \Cref{fig:round1,fig:round2,,fig:round3}.

We now show that \Cref{str:B-box-lmtd} is enough to break the perimetric barrier for the box-limited game.
Note that when using \Cref{str:B-box-lmtd}, Breaker only claims edges from the bounding box of Maker's graph or from its boundary.
Hence, combining \Cref{prop:boxlmtd-unlmted} with the following proposition gives us \Cref{thm:main1}.

\begin{proposition}\label{prop:B-box-lmtd}
    Let $m \ge 36$ and $s \le \tfrac{m-22}{14}$.
    Then by following \Cref{str:B-box-lmtd}, Breaker wins the $(m, 2m - s)$ box-limited percolation game on $\ZZ^2$ in at most three rounds.
\end{proposition}
\begin{proof}
We analyse the game by following \Cref{str:B-box-lmtd} step by step, showing that Breaker can indeed follow it without forfeiting at any point, and thus to win the game by the end of the third round.
In fact, we show that by the end of the third round, Breaker claims all edges in the boundary of the bounding box of Maker's graph, or a subgraph of it containing the origin.

Note that if during the game, Breaker grants Maker with extra edges and wins when playing as if she claimed them, then he also wins the games without granting her those edges.
We will use this assumption as it simplifies the analysis of the game.

Recall that for each $j \ge 1$, we denote by $M_j$ the set of edges that Maker claimed in her $j$-th turn.
Let $m_j \defined \card{M_j}$, so we have $\sum_{j=1}^i m_j \le im$ for any $i \ge 1$.

Refer to \Cref{fig:round1,fig:round2,,fig:round3} for a representation of Rounds 1, 2 and 3 respectively.
We use in several points of this analysis that the game is invariant under translations, rotations by $\pi/4$ angles and horizontal and vertical reflections.

\subsection*{First Round}

Maker plays all her $m_1$ edges $M_1$ in a box-component containing the origin.
Set $B_1 \defined \bb(M_1)$, and let $a_1$ and $b_1$ be the number of vertices in the sides of $B_1$, with $a_1 \geq b_1$.
Assume, without loss of generality, that the top and bottom sides of $B_1$ are at least as large as the left and right ones, that is, they consists of $a_1$ vertices.
Note that as $\card{\boundary B_1} = 2a_1 + 2b_1$, we get
\begin{equation}
\label{eq:a1}
    a_1 \ge \tfrac{1}{4} \card{\boundary B_1}.
\end{equation}

\begin{description}[leftmargin=0pt]
    \item[Step (1)]
    If $\card{\boundary B_1} \le 2m-s$, then by claiming all edges in $\boundary B_1$, Breaker surrounds Maker's graph and wins the game.

    \item[Step (2)]
    Assume otherwise, so we have
    \begin{equation}
    \label{eq:B1bdrylwr}
        \card{\boundary B_1} \ge 2m - s + 1.
    \end{equation}
    Moreover, by \Cref{lem:perim-box-comp} we get
    \begin{equation}
    \label{eq:B1bdryupr}
        \card{\boundary B_1} \le 2m_1 + 4,
    \end{equation}
    so in particular,
    \begin{equation}
    \label{eq:m1lwr}
        m_1 \ge m - \tfrac{1}{2}s - \tfrac{3}{2}.
    \end{equation}
    Denote
    \begin{equation}
    \label{eq:g1}
        g_1 \defined \card{\boundary B_1} - 2m + s.
    \end{equation}
    Breaker chooses $g_1$ boundary-edges in the middle (up to being possibly shifted by one edge, for parity reasons) of the bottom side of $\boundary B_1$, and denotes them by $G_1$.
    We refer to this set of edges as the `gate' for the first round, see \Cref{fig:round1}.
    Then Breaker claims all edges in $(\boundary B_1) \setminus G_1$.

    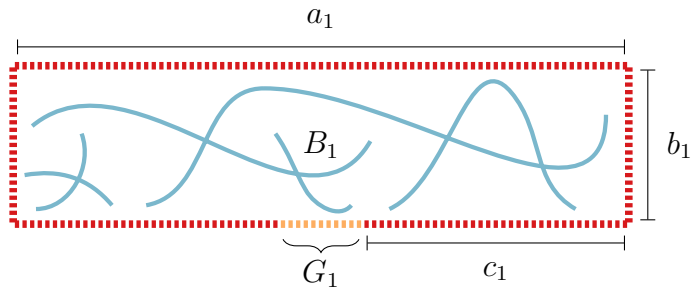
\begin{figure}[ht!]
    \centering
    \begin{tikzpicture}[scale=0.5]
        \tikzstyle{s-helper}=[lightgray, thin];
        \definecolor{s-b-col}{RGB}{215,25,28};
        \tikzstyle{s-b}=[s-b-col, ultra thick];
        \definecolor{s-o-col}{RGB}{253,174,97};
        \tikzstyle{s-o}=[s-o-col, ultra thick];
        \definecolor{s-m-col}{RGB}{122, 183, 204};
        \tikzstyle{s-m}=[s-m-col, ultra thick];
        \definecolor{s-g-col}{RGB}{220,220,220};
        \tikzstyle{s-g}=[s-g-col, line width=0.1pt];
        \clip
            (-10,-2) rectangle (10, 6);
        \foreach \x in {-8,-7.8,...,-1.2}
            \filldraw[xshift=\x cm, s-b] (0,0) -- (0,-0.2);
        \foreach \x in {1.2,1.4,...,8}
            \filldraw[xshift=\x cm, s-b] (0,0) -- (0,-0.2);
        \foreach \x in {-8,-7.8,...,8}
            \filldraw[xshift=\x cm, yshift=4cm, s-b] (0,0) -- (0,0.2);
        \foreach \y in {0,0.2,...,4}
            \filldraw[xshift=-8cm, yshift=\y cm, s-b] (0,0) -- (-0.2,0);
        \foreach \y in {0,0.2,...,4}
            \filldraw[xshift=8cm, yshift=\y cm, s-b] (0,0) -- (0.2,0);
        \draw[draw opacity=0] (-0.5,1.5) rectangle (0.5,2.5)
            node[pos=.5, black] {$B_1$};
        \draw[s-m]
        (1.8,0.3) to[out=25, in=130] (5,3.4)
        (5,3.4) to[out=-50, in=150] (6.7,0.3)
        (-7.6,2.5) to[out=40, in=-120] (1.3,2.1)
        (-4.6,0.4) to[out=10, in=180] (-1.5,3.5)
        (-1.5,3.5) to[out=0, in=-90] (7.5, 2.8)
        (-1.2,2.3) to[out=-55, in=120] (-0.4,0.8)
        (-0.4,0.8) to[out=-60, in=-130] (0.8,0.4)
        (-7.8,1.2) to[out=10, in=130] (-5.5,0.4)
        (-7.5,0.3) to[out=0, in=-70] (-6.3, 2.3);
        \foreach \x in {-1,-0.8,...,0.8,1}
            \filldraw[xshift=\x cm, s-o] (0,0) -- (0,-0.2);
        \draw[decorate,decoration={brace,amplitude=5pt, mirror}]
            (-1,-0.5) -- (1,-0.5) node [midway,below=4pt] {$G_1$};
        \draw [|-|] (-8,4.6) -- (8,4.6) node [midway,above=3pt] {$a_1$};
        \draw [|-|] (8.6,0) -- (8.6,4) node [midway,right=3pt] {$b_1$};
        \draw [|-|] (1.2,-0.6) -- (8,-0.6) node [midway,below=3pt] {$c_1$};
    \end{tikzpicture}
    \caption{End of Round 1, where Maker is in light blue and Breaker in dark red. The set $G_1$ of $g_1$ edges in the gate, in orange, is unclaimed.}
    \label{fig:round1}
    \end{figure}

    This is possible because the bottom side of $\boundary B_1$ contains at least $g_1$ edges, as we observe below.
    \begin{align}
    \label{eq:g1uprs}
        g_1 &\leq s + 4 - 2(m - m_1) \leq s + 4
        && \p*{\text{By (\ref{eq:g1}), (\ref{eq:B1bdryupr}), and $m_1  \leq m$}} \\
        &\leq \frac{1}{4}\p{2m - s + 1} \leq \frac{1}{4}\card{\boundary B_1}
        && \p*{\text{As $s \leq \tfrac{m-22}{14}$, $m \geq 29$ and (\ref{eq:B1bdrylwr})}} \nonumber\\
    \label{eq:g1upr}
        &\leq a_1.
        && \p*{\text{By (\ref{eq:a1})}}
    \end{align}

    Assume further, without loss of generality, that the box is fully contained in the top half-plane and that the origin $(0,0)$ is as close as possible to the centre of the bottom side of the box $B_1$. In particular, it is also in the centre of the gate $G_1$.
\end{description}

\subsection*{Second Round}

First note that by (\ref{eq:g1upr}),
we have $\card{G_1} = g_1 \le a_1 \le m+1 \le 2m - s$.
\begin{description}[leftmargin=0pt]
    \item[Step (1)]
    If $M_2 \cap G_1 = \emptyset$, then by claiming all at most $2m - s$ edges in $G_1$, Breaker surrounds $B_1$ completely and thus wins the game.

    \item[Step (2)]
    Otherwise, we have $M_2 \cap G_1 \neq \emptyset$.
    Let $P_1$, $L_1$ and $R_1$ be as in \Cref{str:B-box-lmtd}, and assume without loss of generality that $L_1$ and $R_1$ are sets of boundary-edges to the left and to the right of $G_1$, respectively.

    Following the strategy, we set $C_1 \defined E(B_1) \cup \boundary B_1$.
    In fact, Breaker can regard the edges in $B_1 \cup G_1$ as being played by Maker, and thus, regard $C_1$ as the set of already claimed edges.
    Note that as $M_2 \cap G_1 \neq \emptyset$, we get $M_2 \cap C_1 \neq \emptyset$.

    Consider the graph spanned by the set of edges $M'_2 \defined (M_2 \cup P_1) \setminus C_1$.
    As this is a proper subset of $M_2 \cup P_1$, it might not be box-connected.
    Consider its box-component containing $P_1$, and assume without loss of generality that it is $M'_2$ itself, as otherwise it has only less edges.
    Indeed, we can `return' the edges outside the box-component of $P_1$ back to Maker, since they will not be claimed by Breaker in this turn and Maker can reclaim them immediately in the next turn if she wishes to do so.

    Denote by $B_2 \defined \bb(M'_2)$ the new bounding box.
    Let $a_2$ and $b_2$ be the number of vertices in the top and bottom sides and in the left and right sides of $B_2$, respectively.
    Since $M_2 \cap G_1 \neq \emptyset$, we have $\card{M_2 \setminus C_1} \le m_2 - 1$. Furthermore,
    \begin{equation*}
        \card{P_1} = \card{V_1} - 1 = \card{G_1} - 1 = g_1 - 1,
    \end{equation*}
    so we have
    \begin{equation*}
        \card{M'_2} \le m_2 + g_1 - 2.
    \end{equation*}
    Therefore, \Cref{lem:perim-box-comp} implies that
    \begin{equation}
    \label{eq:B2bdry}
        \card{\boundary B_2} = 2a_2 + 2b_2 \le 2m_2 + 2g_1.
    \end{equation}
    Moreover, as $P_1 \subseteq B_2$, we get
    \begin{equation}
    \label{eq:g1a2}
        g_1 \leq a_2.
    \end{equation}

    \item[Step (2.1)]
    If $\card{(\boundary B_2) \setminus C_1} \le 2m - s$, then Breaker claims all edges in $(\boundary B_2) \setminus C_1$.
    Note that $G_1 \subseteq \boundary B_1 \cap (E(B_2) \cup \boundary B_2)$.
    Thus, Breaker surrounds $B_1 \cup B_2$ completely and wins the game.

    \item[Step (2.2)]
    Assume otherwise, and denote
    \begin{equation}
    \label{eq:g2}
        g_2 \defined \card{(\boundary B_2) \setminus C_1} - 2m + s.
    \end{equation}
    First, note that by (\ref{eq:B2bdry}) we get
    \begin{equation}
    \label{eq:g2upr}
        g_2 \le 2g_1 + s + 2(m_2 - m).
    \end{equation}
    Combining the assumption that $g_2 \geq 1$ and (\ref{eq:B2bdry}),
    it follows that
    \begin{equation*}
        2m - s + 1 \le \card{(\boundary B_2) \setminus C_1} \le 2m_2 + 2g_1 - \card{(\boundary B_2) \cap C_1},
    \end{equation*}
    and hence
    \begin{equation}
    \label{eq:B2projB1}
        \card{(\boundary B_2) \cap C_1} \le 2g_1 + s - 1 + 2(m_2 - m).
    \end{equation}

    As $L_1$ and $R_1$ are sets of boundary-edges of $B_1$ on both sides of $G_1$ claimed by Breaker, define $c_1 \defined \min \set{\card{L_1},\card{R_1}}$, and we have
    \begin{align}
        c_1  &= \floor{\tfrac{1}{2} (a_1 - g_1)}
        \geq \tfrac{1}{2}(a_1 - g_1 - 1) \nonumber \\
        &\ge \tfrac{1}{2} \p[\big]{\tfrac{1}{4}\card{\boundary B_1} - \card{\boundary B_1} + 2m-s -1 }
        && \p{\text{By (\ref{eq:a1}) and (\ref{eq:g1})}} \nonumber \\
        &= \tfrac{1}{2} \p[\big]{2m - s - 1 - \tfrac{3}{4}\card{\boundary B_1} } \nonumber \\
        &\ge \tfrac{1}{2} \p[\big]{2m - s - 1 - \tfrac{3}{4}(2m+4)}
        && \p{\text{By (\ref{eq:B1bdryupr}) and $m_1 \le m$}} \nonumber \\
    \label{eq:c1}
        &= \tfrac{1}{4}m - \tfrac{1}{2}s - 2.
    \end{align}

    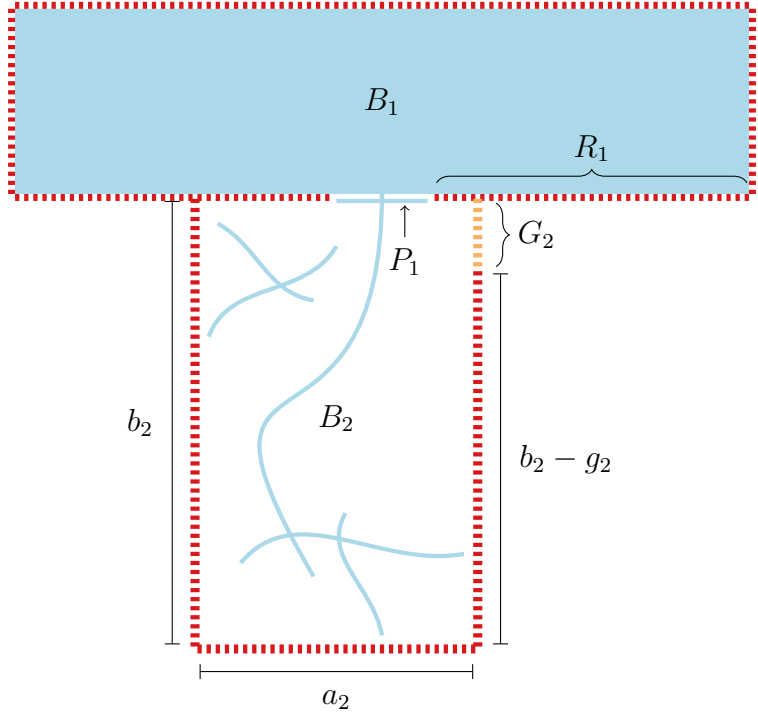
\begin{figure}
    \centering
    \begin{tikzpicture}[scale=0.6]
        \tikzstyle{s-helper}=[lightgray, thin];
        \definecolor{s-b-col}{RGB}{215,25,28};
        \tikzstyle{s-b}=[s-b-col, ultra thick];
        \definecolor{s-o-col}{RGB}{253,174,97};
        \tikzstyle{s-o}=[s-o-col, ultra thick];
        \definecolor{s-m-col}{RGB}{171,217,233};
        \tikzstyle{s-m}=[s-m-col, ultra thick];
        \definecolor{s-g-col}{RGB}{220,220,220};
        \tikzstyle{s-g}=[s-g-col, line width=0.1pt];
        \clip
            (-9,-11.6) rectangle (9, 4.6);
        \foreach \x in {-8,-7.8,...,-1.2}
            \filldraw[xshift=\x cm, s-b]
                (0,0) -- (0,-0.2);
        \foreach \x in {1.2,1.4,...,8}
            \filldraw[xshift=\x cm, s-b]
                (0,0) -- (0,-0.2);
        \foreach \x in {-8,-7.8,...,8}
            \filldraw[xshift=\x cm, yshift=4cm, s-b]
                (0,0) -- (0,0.2);
        \foreach \y in {0,0.2,...,4}
            \filldraw[xshift=-8cm, yshift=\y cm, s-b]
                (0,0) -- (-0.2,0);
        \foreach \y in {0,0.2,...,4}
            \filldraw[xshift=8cm, yshift=\y cm, s-b]
                (0,0) -- (0.2,0);
        \foreach \y in {-0.2,-0.4,...,-10}
            \filldraw[xshift=-4cm, yshift=\y cm, s-b]
                (0,0) -- (-0.2,0);
        \foreach \x in {-4,-3.8,...,2}
            \filldraw[xshift=\x cm, yshift=-10cm, s-b]
                (0,0) -- (0,-0.2);
        \foreach \y in {-1.8,-2.0,...,-10}
            \filldraw[xshift=2cm, yshift=\y cm, s-b]
                (0,0) -- (0.2,0);
        \filldraw[s-m, step=0.2] (-8,0) rectangle (8,4);
        \filldraw[s-m] (-0.5,1.5) rectangle (0.5,2.5)
            node[pos=.5, black] {$B_1$};
        \foreach \y in {-0.2,-0.4,...,-1.6}
            \draw[yshift=\y cm, xshift=2 cm, s-o] (0,0) -- (0.2,0);
        \draw[s-m] (0,0) -- (0,-0.2);
        \draw[s-m] (-1,-0.2) -- (1,-0.2);
        \draw[s-m]
            (0,0) to[out=-90,in=120, looseness=2] (-1.5, -8.5)
            (-3.1,-8.2) to[out=50, in=-170] (1.8,-8)
            (-0.8,-7.1) to[out=-120, in=100] (0
            ,-9.8)
            (-3.6,-0.7) to[out=-30, in=170] (-1.5,-2.4)
            (-3.8,-3.2) to[out=70,in=-120] (-1,-1.2)
        ;
        \draw[draw opacity=0] (-1.5,-4.5) rectangle (-0.5,-5.5)
            node[pos=.5, black] {$B_2$};
        \draw [decorate,decoration={brace,amplitude=5pt}]
          (2.5,-0.25) -- (2.5,-1.65) node [midway,right=4pt] {$G_2$};
        \draw [|-|]
          (-4,-10.6) -- (2,-10.6) node [midway,below=3pt] {$a_2$};
        \draw [|-|]
          (-4.6,-0.2) -- (-4.6,-10) node [midway,left=3pt] {$b_2$};
        \draw [|-|]
          (2.6,-1.8) -- (2.6,-10) node [midway,right=3pt] {$b_2 - g_2$};
        \draw [decorate,decoration={brace,amplitude=5pt}]
          (1.2,0.2) -- (8,0.2) node [midway,above=4pt] {$R_1$};
        \draw [<-]  (0.5,-0.3) -- (0.5,-1) node[below] {$P_1$};
    \end{tikzpicture}
    \caption{End of Round 2, where Maker is in light blue and Breaker in dark red. In orange, the set $G_2$ of $g_2$ consecutive unclaimed edges, forms the gate for the second round.}
    \label{fig:round2}
    \end{figure}

    Let $x_0$ be minimal such that
    \begin{equation*}
        \abs{x_0} \geq \max \set*{ \abs{x} \st (x,y) \in V \p*{ \ZZ^2[\boundary B_1]} }.
    \end{equation*}
    Assume that Maker claimed any edge which is either to the right or to the left of $B_1$, that is, an edge which has a vertex $v = (x_0,y_0)$.
    In particular, as Maker plays box-connected, we get that $B_2 \setminus C_1$ contains a horizontal path of length at least $c_1 + g_1$, and thus $(\boundary B_2) \cap C_1$ contains such a path as well.
    Hence, we get that
    \begin{align*}
        \card{(\boundary B_2) \cap C_1} &\ge c_1 + g_1 \\
        &\ge \tfrac{1}{4}m - \tfrac{1}{2}s - 2 + g_1
        && \p{\text{By (\ref{eq:c1})}} \\
        &\ge 2s + 4 + g_1
        && \p[\big]{\text{As $s  \leq \tfrac{m - 22}{14}$ and $m \geq 29$}}\\
        &\geq 2s + 4 - 2(m - m_1) + 2(m_2 - m) + g_1
        && \p{\text{By (\ref{eq:misum})}} \\
        &\ge 2g_1 + s + 2(m_2 - m),
        && \p{\text{By (\ref{eq:g1})}}
    \end{align*}
    contradicting (\ref{eq:B2projB1}).
    Hence, Maker did not claim any such edge.

    It follows that $B_2$ is completely contained in the infinite vertical stripe defined by the right and left sides of $B_1$.
    We remark that requiring $m \ge 29$ is necessary for this very step.
    More formally, we have
    \begin{equation}
    \label{eq:a2}
        a_2 = \card{(\boundary B_2) \cap C_1}.
    \end{equation}
    Furthermore, note that we must have $B_1 \cap B_2 = \emptyset$, which implies $G_1 \subseteq \boundary B_2$.
    In particular, by (\ref{eq:B2projB1}), we get
    \begin{equation}
    \label{eq:a2upr}
         a_2 \le 2g_1 + s - 1 + 2(m_2 - m),
    \end{equation}
    and by (\ref{eq:g2}), we get that
    \begin{equation}
    \label{eq:g2precise}
        g_2 = \card{\boundary B_2} - a_2 - 2m + s.
    \end{equation}
    Recall that $\card{R_1 \cap (\boundary B_2)} \le \card{L_1 \cap (\boundary B_2)}$, so in total we get
    \begin{equation}
    \label{eq:a1-a2}
        \card{R_1 \setminus (\boundary B_2)} \ge \tfrac{1}{2} (a_1 - a_2).
    \end{equation}

    Let $G_2$ be a set of $g_2$ consecutive edges starting at an edge $e\in (\boundary B_2) \setminus C_1$ of minimal distance to $G_1$.
    Note that $e$ is in fact the top-right horizontal edge in $\boundary B_2$, and thus $G_2$ the set of $g_2$ most top boundary-edges on the right side of $B_2$.
    Following \Cref{str:B-box-lmtd}, Breaker claims all $2m - s$ edges of $(\boundary B_2) \setminus C_1$, leaving the edges of $G_2$ unclaimed as in \Cref{fig:round2}, which we refer as the gate for the second round.
\end{description}

\subsection*{Third Round}

We now show that there exists a set of at most $2m-s$ unclaimed edges such that by claiming them, Breaker wins the game in this round.
We do this by a bit more careful geometric analysis, and by combining together all the information described above in the previous two rounds of the game.

We start by giving some notations analogously to those in the first two rounds in \Cref{str:B-box-lmtd}.
Let $M_3$ be the set of edges that Maker claimed on her third turn.
So we have $m_3 \defined \card{M_3}$.
Let $C_2 \defined C_1 \cup E(B_2) \cup \boundary B_2$.
Again, we sometimes regard the edges in $B_2 \cup G_2$ as being played by Maker, and thus regard $C_2$ as the set of already claimed edges.
Denote by $D$ the set of boundary-edges adjacent to the right side of $B_2$.
So we have $\card{D} = b_2$.
Denote by $A$ the box for which $D$ is the set of boundary-edges of the left side of it, and $R_1 \setminus (\boundary B_2)$ is the set of boundary-edges of the top side of it (see \Cref{fig:round3}).
We consider three cases.

\begin{description}[leftmargin=0pt]
    \item[Case 1]
    $M_3 \cap G_2 = \emptyset$.

    Similarly to the second round, by (\ref{eq:g2upr}) and (\ref{eq:g1uprs}), and the fact that $m_1 + m_2 \le 2m$, we have $\card{G_2} = g_2 \le 3s + 8 \le 2m - s$.
    Thus in this case we have $\card{M_3 \cap G_2} \le 2m-s$, and by claiming all edges in the gate $G_2$, Breaker surrounds $B_1 \cup B_2$ completely and wins the game.

    \item[Case 2]
    $M_3 \cap G_2 \neq \emptyset$ and $(M_3 \setminus C_2) \cap (\boundary A) = \emptyset$.

    In this case, by claiming all unclaimed edges in $\boundary A$, Breaker surrounds Maker's graph completely and wins the game.
    We show that he can indeed do so.

    Note that there are $\card{\boundary A} = 2\card{D} + 2\card{R_1\setminus (\boundary B_2)}$ edges in the boundary of $A$, from which only $\card{D} + \card{R_1\setminus (\boundary B_2)}$ are unclaimed.
    We get that,
    \begin{align}
        \card{(\boundary A) \setminus C_2} &\le \card{D} + \card{R_1 \setminus (\boundary B_2)} \nonumber\\
        &\le \card{D} + \tfrac{1}{2}m_1
        && \p{\text{As $\card{R_1} \le \tfrac{1}{2}m_1$}} \nonumber\\
        &\le m_2 + \tfrac{1}{2}m_1
        && \p{\text{As $\card{D} = b_2 \leq m_2$}} \nonumber\\
        &\le 2m - \tfrac{1}{2}m_1
        && \p{\text{As $m_1 + m_2 \le 2m$}} \nonumber\\
    \label{eq:case2}
        &\le 2m - s,
        && \p{\text{As $m_1 \geq 2s$ by (\ref{eq:m1lwr})}}
    \end{align}
    and therefore, Breaker wins the game.

    \begin{figure}
    \centering
    \begin{tikzpicture}[scale=0.6]
        \tikzstyle{s-helper}=[lightgray, thin];
        \definecolor{s-b-col}{RGB}{215,25,28};
        \tikzstyle{s-b}=[s-b-col, ultra thick];
        \definecolor{s-o-col}{RGB}{253,174,97};
        \tikzstyle{s-o}=[s-o-col, ultra thick];
        \definecolor{s-m-col}{RGB}{171,217,233};
        \tikzstyle{s-m}=[s-m-col, ultra thick];
        \definecolor{s-g-col}{RGB}{220,220,220};
        \tikzstyle{s-g}=[s-g-col, line width=0.1pt];
        \clip
            (-9,-12.2) rectangle (10, 5.6);
        \foreach \x in {-8,-7.8,...,-1.2}
            \filldraw[xshift=\x cm, s-b]
                (0,0) -- (0,-0.2);
        \foreach \x in {1.2,1.4,...,8}
            \filldraw[xshift=\x cm, s-b]
                (0,0) -- (0,-0.2);
        \foreach \x in {-8,-7.8,...,8}
            \filldraw[xshift=\x cm, yshift=4cm, s-b]
                (0,0) -- (0,0.2);
        \foreach \y in {0,0.2,...,4}
            \filldraw[xshift=-8cm, yshift=\y cm, s-b]
                (0,0) -- (-0.2,0);
        \foreach \y in {0,0.2,...,4}
            \filldraw[xshift=8cm, yshift=\y cm, s-b]
                (0,0) -- (0.2,0);
        \foreach \y in {-0.2,-0.4,...,-10}
            \filldraw[xshift=-4cm, yshift=\y cm, s-b]
                (0,0) -- (-0.2,0);
        \foreach \x in {-4,-3.8,...,2}
            \filldraw[xshift=\x cm, yshift=-10cm, s-b]
                (0,0) -- (0,-0.2);
        \foreach \y in {-1.8,-2.0,...,-10}
            \filldraw[xshift=2cm, yshift=\y cm, s-b]
                (0,0) -- (0.2,0);
        \foreach \y in {-0.2,-0.4,...,-9.8}
            \filldraw[xshift=7cm, yshift=\y cm, s-b]
                (0,0) -- (0.2,0);
        \foreach \x in {3.6,3.8,4.0,4.2,4.4,4.6,4.8,5.0}
            \filldraw[xshift=\x cm, yshift=-10cm, s-b]
                (0,0) -- (0,-0.2);
        \foreach \x in {-5.4,-5.2,...,3.4}
            \filldraw[xshift=\x cm, yshift=-11.2cm, s-b]
                (0,0) -- (0,-0.2);
        \foreach \y in {-9.0,-9.2,...,-11.2}
            \filldraw[xshift=-5.6cm, yshift=\y cm, s-b]
                (0,0) -- (0.2,0);
        \foreach \x in {-5.4,-5.2,...,-4.2}
            \filldraw[xshift=\x cm, yshift=-8.8cm, s-b]
                (0,0) -- (0,-0.2);
        \foreach \y in {-10.2,-10.4,...,-11.2}
            \filldraw[xshift=3.4cm, yshift=\y cm, s-b]
                (0,0) -- (0.2,0);
        \foreach \y in {-10.2,-10.4,...,-11.8}
            \filldraw[xshift=5.0cm, yshift=\y cm, s-b]
                (0,0) -- (0.2,0);
        \foreach \x in {5.2,5.4,...,7.6}
            \filldraw[xshift=\x cm, yshift=-12.0cm, s-b]
                (0,0) -- (0,0.2);
        \foreach \y in {-10,-10.2,...,-11.8}
            \filldraw[xshift=7.6cm, yshift=\y cm, s-b]
                (0,0) -- (0.2,0);
        \foreach \x in {7.2,7.4,7.6}
            \filldraw[xshift=\x cm, yshift=-10.0cm, s-b]
                (0,0) -- (0,0.2);
        \filldraw[s-m-col] (-8,0) rectangle (8,4);
        \filldraw[s-m] (-0.5,1.5) rectangle (0.5,2.5)
            node[pos=.5, black] {$B_1$};
        \filldraw[s-m-col, step=0.2] (-4,-0.2) rectangle (2,-10);
        \filldraw[s-m] (-1.5,-4.5) rectangle (-0.5,-5.5)
            node[pos=.5, black] {$B_2$};
        \draw[draw opacity=0] (4.1,-4.5) rectangle (5.1,-5.5)
            node[pos=.5, black] {$A'$};
        \draw[draw opacity=0] (-5.5,-10.3) rectangle (-4.5,-11.3)
            node[pos=.5, black] {$S_1$};
        \draw[draw opacity=0] (6.5,-10.3) rectangle (7.5,-11.3)
            node[pos=.5, black] {$S_2$};
        \draw[s-m] (0,0) -- (0,-0.2);
        \draw[s-m] (2,-1) -- (2.2,-1);
        \draw[s-m] (2.2,-0.2) -- (2.2,-1.6);
        \draw[s-m]
            (2.2,-1) to[out=0,in=90]
            (3.2,-10) -- (3.2,-10.2)
            to[out=-100, in=0] (-1,-10.7)
            to[out=180, in=-110] (-5.2,-9.2)
            (5.3,-5.8) to[out=100, in=-80] (6.8,-3.7)
            (5.4,-3.7) to[out=-100, in=80] (6.3,-5.5)
            (3.0,-7.8) to[out=70, in=-110] (6.2,-8.2)
            (6,-7.6) to[out=-150, in=90] (5.8,-10) --
            (5.8,-10.2) to[out=-90, in=170] (7.5,-11.7)
            (5.3,-11.7) to[out=0,in=-140] (6.0,-11.3)
        ;
        \draw [decorate,decoration={brace,amplitude=5pt, mirror}]
          (1.7,-0.25) -- (1.7,-9.95) node [midway,left=4pt] {$D$};
        \draw [|-|]
          (8.6,-0.2) -- (8.6,-10) node [midway,right=3pt] {$b_3$};
        \draw [|-|]
          (2.25,-0.5) -- (6.95,-0.5) node [midway,below=3pt] {$a_3$};
        \draw [decorate,decoration={brace,amplitude=5pt}]
          (2.4,0.2) -- (7.95,0.2) node [midway,above=4pt] {$R_1 \setminus (\partial B_2)$};
        \draw [|-|] (-8,4.6) -- (8,4.6) node [midway,above=3pt] {$a_1$};
        \draw [|-|]
          (-3.95,0.3) -- (1.95,0.3) node [midway,above=3pt] {$a_2$};
        \draw [densely dashed] (2.2,-10) -- (8,-10) -- (8,-0.2);
        \draw [<-]  (2.4,-1.4) -- (2.9,-1.9) node[right=5pt,below=0pt] {$P_2$};
    \end{tikzpicture}
    \caption{End of Round 3, where Maker is in light blue and Breaker in dark red. The region bounded by the dashed lines is the box $A$.}
    \label{fig:round3}
    \end{figure}
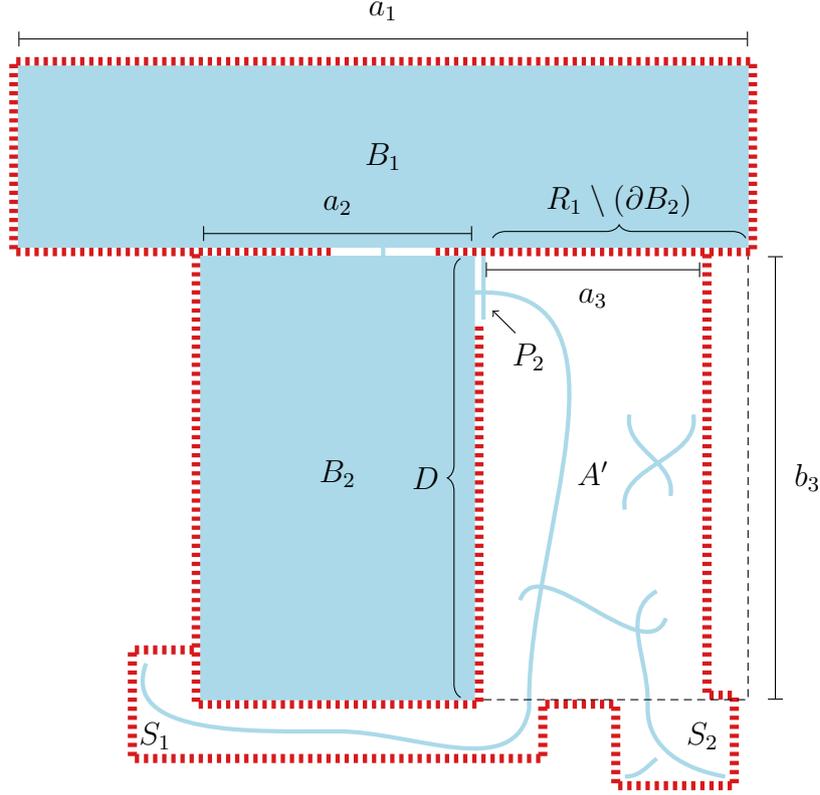

    \item[Case 3]
    $M_3 \cap G_2 \neq \emptyset$ and $(M_3 \setminus C_2) \cap (\boundary A) \neq \emptyset$.

    For this case we need some further notation.
    Similarly to the second round, let $V_2$ be the set of vertices in $\ZZ^2 \setminus B_2$ which are contained in edges of $G_2$.
    Let further $P_2 \defined E \p*{\ZZ^2[V_2]}$ be the set of edges in the path induced by the vertices of $V_2$.
    Note that we have $P_2 \subseteq A$.
    Let $M'_3 \defined (M_3 \cup P_2)\setminus C_2$.
    Similarly to the second round, we have $\card{P_2} = g_2 - 1$, and since $M_3 \cap G_2 \neq \emptyset$, we also have $\card{M_3 \setminus C_2} \le m_3 - 1$.
    In total we get
    \begin{equation}
    \label{eq:M3}
        \card{M'_3} \le m_3 + g_2 - 2.
    \end{equation}

    Let $M_3^A \defined M'_3 \cap E(A)$, and let $A' \defined \bb(M_3^A)$ be its bounding-box (see \Cref{fig:round3}).
    By box-connectivity, and since $(M_3 \setminus C_2) \cap (\boundary A) \neq \emptyset$, we get that the box $A'$ shares at least one full side with the box $A$.
    In \Cref{fig:round3}, they share the left side, for instance.

    If $(\boundary A') \cap M'_3 = \emptyset$, then Breaker wins simply by claiming all the edges in $\boundary A' \setminus C_2$.
    Indeed, as $A' \subseteq A$ and (\ref{eq:case2}), we have
    \begin{align*}
        \card{(\boundary A') \setminus C_2} \le 2m - s.
    \end{align*}

    Hence we may assume that $(\boundary A') \cap M'_3 \neq \emptyset$.
    Let $S_1,\dotsc,S_t$ be the box-components of $M'_3 \setminus M_3^A$ such that intersect the boundary of $A'$.
    That is, for $i \in [t]$ we have $S_i \cap (\boundary A') \neq \emptyset$.
    Further, denote $Q_i \defined \bb(S_i)$.

    We now show that Breaker can claim all unclaimed edges that surround
    \begin{equation*}
        A' \cup Q_1 \cup \dotsb \cup Q_t,
    \end{equation*}
    which, in particular, implies that he wins the game.
    More precisely, we show that there at most $2m-s$ such edges, that is,
    \begin{equation}
    \label{eq:BoxesBdry}
        \card[\big]{\boundary \p[\big]{A' \cup Q_1 \cup \dotsb \cup Q_t} \setminus C_2} \le 2m-s.
    \end{equation}

    Let us emphasise that in the equation above we mean the edge-boundary of a union of boxes, rather than the edge-boundary of their joint bounding box, as in previous rounds.

    To prove (\ref{eq:BoxesBdry}), first note that for any distinct $i,j \in [t]$, we have $Q_i \cap Q_j = \emptyset$ by box-connectivity.
    Therefore,
    \begin{align}
        \card[\big]{\boundary \p[\big]{A' \cup Q_1 \cup \dotsb \cup Q_t} \setminus C_2} &\le \card[\big]{(\boundary A')\setminus C_2} + \sum_{i=1}^t \card[\big]{\boundary (A' \cup Q_i) \setminus \boundary A'} \nonumber \\
    \label{eq:SnakesBdry}
        &\le \card[\big]{(\boundary A')\setminus C_2} + \sum_{i=1}^t \p[\big]{\card{\boundary (A' \cup Q_i)} - \card{\boundary A'}}.
    \end{align}
    Now we bound each of these two terms separately.

    We start with bounding the sum in (\ref{eq:SnakesBdry}).
    Similarly to the argument in the proof of \Cref{lem:perim-box-comp}, for each $i\in [t]$ we have
    \begin{equation*}
        \card{\boundary(A' \cup Q_i)} \le \card{\boundary A'} + \card{\boundary Q_i} - 4.
    \end{equation*}
    Recall that by \Cref{lem:perim-box-comp}, we have $\card{\boundary Q_i} \le 2e(S_i) + 4$ for each $i\in [t]$, so in total we get
    \begin{equation}
    \label{eq:SnakeSum}
        \sum_{i=1}^t \p[\big]{\card{\boundary (A' \cup Q_i)} - \card{\boundary A'}} \le 2\sum_{i=1}^t e(S_i).
    \end{equation}

    As for the first term in (\ref{eq:SnakesBdry}), let $a_3$ and $b_3$ be the numbers of vertices in the top and bottom sides and in the left and right sides of $A'$, respectively.
    Recall that $A'$ shares at least one side with $A$, so in particular we have either $a_3 = \card{R_1 \setminus (\boundary B_2)}$ or $b_3 = b_2$.
    It follows that there are at least $\min \set{ \card{R_1 \setminus (\boundary B_2)}, b_2 - g_2 }$ edges in $\boundary A'$ which are already claimed by Breaker.
    Moreover, we have $G_2 \subseteq \boundary A'$, so we get a reduction of at least $g_2$ more edges from the amount that Breaker has to claim in $\boundary A'$.
    In total, we get
    \begin{equation*}
        \card[\big]{(\boundary A')\cap C_2} \ge g_2 + \min \set[\big]{ \card{R_1 \setminus (\boundary B_2 )} ,\, b_2 - g_2}.
    \end{equation*}

    Denote $m^A_3 \defined \card{M^A_3}$ and recall that $A' = \bb(M^A_3)$.
    By \Cref{lem:perim-box-comp} and by the above, we get that
    \begin{align*}
        \card[\big]{(\boundary A' )\setminus C_2} &= \card[\big]{\boundary A'} - \card[\big]{(\boundary A')\cap C_2} \\
        &\le 2 m_3^A + 4 - g_2 - \min \set[\big]{ \card{R_1 \setminus (\boundary B_2 )} ,\, b_2 - g_2}.
    \end{align*}
    To finish the proof, we need the following technical claim, which we prove later.

    \begin{claim}
    \label{claim:min}
        We have
        \begin{equation}
        \label{eq:round3min}
            \min \set[\big]{ \card{R_1 \setminus (\boundary B_2 )} ,\, b_2 - g_2} \ge g_2 + s + 2(m_3 - m).
        \end{equation}
    \end{claim}

    Assume for now that \Cref{claim:min} holds.
    We get that
    \begin{equation}
    \label{eq:A'Bdry}
        \card[\big]{(\boundary A') \setminus C_2} \le 2(m + m^A_3 - m_3) - 2g_2 - s + 4.
    \end{equation}
    Furthermore, by (\ref{eq:M3}) we have
    \begin{align*}
        m_3 + g_2 - 2 \ge \card{M'_3} \ge m^A_3 + \sum_{i=1}^t e(S_i),
    \end{align*}
    and in particular
    \begin{align*}
        m^A_3 - m_3 \le g_2 - 2 - \sum_{i=1}^t e(S_i).
    \end{align*}
    So by (\ref{eq:A'Bdry}), we get
    \begin{align*}
        \card[\big]{(\boundary A') \setminus C_2} \le 2 \p[\Big]{m + g_2 - 2 - \sum_{i=1}^t e(S_i)} - 2g_2 - s + 4 = 2m - s - 2\sum_{i=1}^t e(S_i).
    \end{align*}
    Finally, by (\ref{eq:SnakeSum}), (\ref{eq:SnakesBdry}), and by the above, we conclude
    \begin{align*}
        \card[\Big]{\boundary \p[\big]{A' \cup Q_1 \cup \dotsb \cup Q_t} \setminus C_2} \leq 2m - s,
    \end{align*}
    proving (\ref{eq:BoxesBdry}), as required.

    Hence, it is only left to prove \Cref{claim:min}.

    \begin{proof}[Proof of \Cref{claim:min}]
        We start with the second term.
        Observe that
        \begin{align*}
            g_2 &= a_2 + 2b_2 - 2m + s
            && \p{\text{By (\ref{eq:g2precise}) and (\ref{eq:B2bdry})}}\\
            &\le b_2 + g_1 + s - (2m - m_2).
            && \p{\text{Again by (\ref{eq:B2bdry})}}
        \end{align*}
        Therefore, we have
        \begin{align}
            b_2 - g_2 &\ge 2m - m_2 - s - g_1 \nonumber\\
            &\ge 2m - m_2 - 2s - 4 + 2(m - m_1)
            && \p{\text{By (\ref{eq:g1uprs})}}\nonumber\\
        \label{eq:min2large}
            &\ge m - 2s - 4,
        \end{align}
        where we used the fact that $m_1 \leq m$ and $m_1 + m_2 \leq 2m$.

        On the other hand, we have
        \begin{align*}
            g_2 + s + 2(m_3 - m) &\le 2g_1 + 2s + 2(m_2 - m) + 2(m_3 - m)
            && \p{\text{By (\ref{eq:g2upr})}} \\
            &\le 4s + 8 + 4m_1 + 2(m_2 + m_3) - 8m
            && \p{\text{By (\ref{eq:g1uprs})}} \\
            &\le 4s + 8
            && \p{\text{By (\ref{eq:misum})}} \\
            &\le m - 2s - 4.
            && \p[\big]{\text{As $s \le \tfrac{m-22}{14}$}}.
        \end{align*}
        Combining this with (\ref{eq:min2large}), we get
        \begin{equation*}
            b_2 - g_2 \ge g_2 + s + 2(m_3 - m),
        \end{equation*}
        proving the claim for the second term in~(\ref{eq:round3min}).

        As for the first term, we start by noticing that
        \begin{align}
            \tfrac{1}{2}a_1 + \tfrac{1}{2}a_2 - 2g_1 &\ge \tfrac{1}{2} (a_1 - 3g_1)
            && \p{\text{By (\ref{eq:g1a2})}} \nonumber\\
            &\ge \tfrac{1}{2} \p[\Big]{\tfrac{1}{4}\card{\boundary B_1} - 3\p[\big]{\card{\boundary B_1} - 2m + s } }
            && \p{\text{By (\ref{eq:a1}) and (\ref{eq:g1})}} \nonumber \\
            &= \tfrac{1}{2}\p[\big]{6m - 3s - \tfrac{11}{4}\card{\boundary B_1} } \nonumber \\
        \label{eq:a1a2g1}
            &\ge \tfrac{1}{2}\p[\big]{6m - 3s - \tfrac{11}{4}(2m_1+4)}.
            && \p{\text{By (\ref{eq:B1bdryupr})}}
        \end{align}
        In addition, using (\ref{eq:g2precise}) and (\ref{eq:B2bdry}), we can also write
        \begin{equation}
        \label{eq:g2bnd}
            g_2 \le 2g_1 + s - a_2 - 2(m - m_2).
        \end{equation}
        Hence we get
        \begin{align*}
            \card{R_1 \setminus (\boundary B_2 )} - g_2&\ge  \tfrac{1}{2} (a_1 - a_2) - g_2
            && \p{\text{By (\ref{eq:a1-a2})}}\\
            &\ge \tfrac{1}{2}a_1 + \tfrac{1}{2}a_2 - 2g_1 - s + 2(m - m_2)
            && \p{\text{By (\ref{eq:g2bnd})}}\\
            &\ge \tfrac{1}{2}\p[\big]{6m - 3s - \tfrac{11}{4}(2m_1+4)} - s + 2(m - m_2)
            && \p{\text{By (\ref{eq:a1a2g1})}}\\
            &= 7m - 2(m_1 + m_2 + m_3) \\
            &\quad- \tfrac{3}{4}m_1 - \tfrac{5}{2}s - \tfrac{11}{2} + 2(m_3 - m) \\
            &\ge m - \tfrac{3}{4}m_1 - \tfrac{5}{2}s - \tfrac{11}{2} + 2(m_3 - m)
            && \p{\text{By (\ref{eq:misum})}} \\
            &\ge \tfrac{1}{4}m - \tfrac{5}{2}s - \tfrac{11}{2} + 2(m_3 - m)
            && \p{\text{As $m_1 \leq m$}} \\
            &\ge s + 2(m_3 - m)
            && \p[\big]{\text{As $s \leq \tfrac{m-22}{14}$}},
        \end{align*}
        finishing the proof of the claim.
    \end{proof}
\end{description}
This completes the proof.
\end{proof}

Note that a more careful analysis, in similar spirit to the one above, might produce a somewhat smaller ratio than $2-\frac{1}{14}+o(1)$ in \Cref{thm:main1}.
However, as our main aim was to break the perimetric barrier, we did not attempt to optimise it in the benefit of clarity.


\section{Fast win of Breaker in the boosted \texorpdfstring{$(m,2m)$}{(m,2m)}-game}
\label{sec:fast-boost}

In this section we prove \Cref{thm:boost}.
Note that, as the game is bias monotone in $b$, it is enough to prove \Cref{thm:boost} for $b = 2m$.

As discussed in \Cref{sec:intro}, the proof contains two important ideas.
Firstly, assuming that Maker plays connected significantly simplifies the analysis of the game.
Hence we consider a slight variation of the game for which the following hold.
\begin{enumerate}
    \item If Breaker wins this variant within the first $k$ rounds using a strategy satisfying certain further condition, then he also wins the $(m,b)$ percolation game within the first $k$ rounds.
    \item Maker must keep the graph spanned by the edges she claimed connected.
\end{enumerate}

Note that there is a certain price we have to pay for this in the form of restrictions on Breaker's strategy; and in allowing Maker, instead of claiming precisely $m$ edges in each round, to claim sometimes a bit more and sometimes a bit less.

The second important idea is to define a strategy of Breaker in such a way that he forces Maker to create a component of a `bad' shape.
More precisely, as we can see by \Cref{lem:perimetric}, if we give Maker $km + c$ edges to build a connected component and Breaker $2km$ edges to place in the boundary of the said component, Breaker can ensure to claim `almost' all edges in the boundary.
Hence, if he can force Maker to create a component that is sufficiently far from equality in \Cref{lem:perimetric}, Breaker can in fact claim all the edges in its boundary.
In the rest of the section, we expand these ideas into a formal proof.

We start by defining the variation of the game that we mentioned before.

\begin{definition}[$c$-boosted $(m,b)$ Maker-Breaker limited percolation game on $\ZZ^2$]
\label{def:bstd-lmtd-game}
    Two players, Maker and Breaker alternate claiming yet unclaimed edges of a board $\ZZ^2$, starting in round $1$ with Maker going first.
    \begin{itemize}
        \item In round $i$, Maker chooses a non-negative integer $m_i$ such that
        \begin{equation*}
            \sum_{j=1}^i m_j \leq im + c,
        \end{equation*}
        and then claims $m_i$ unclaimed edges from $E(\ZZ^2)$.
        Moreover, Maker must play in a way that in the end of each of her turns, her edges must be in the component of $v_0$.
        \item In each round, Breaker claims at most $b$ unclaimed edges.
        \item Breaker wins if the connected component of $v_0$ consisting only of Maker's edges and unclaimed edges becomes finite. If Maker can ensure that this never happens, then she wins.
    \end{itemize}
\end{definition}

The following proposition is a key result of the same spirit as \Cref{prop:boxlmtd-unlmted}, relating these two games.

\begin{proposition}\label{prop:lmtd-unlmtd}
Let $m, b \ge 1$ and $c \ge 0$ be integers.
Assume that Breaker can ensure his win in the $c$-boosted $(m,b)$ limited percolation game on $\ZZ^2$ within the first $k$ rounds by claiming only edges from the boundary of the graph spanned by Maker's edges.
Then he can also ensure his win in the $c$-boosted $(m,b)$ percolation game on $\ZZ^2$ within the first $k$ rounds.
\end{proposition}

The proof of \Cref{prop:lmtd-unlmtd} is similar to the proof of \Cref{prop:boxlmtd-unlmted} except for some minor details, so we do not include it.

Throughout the rest of the section, we consider only the $c$-boosted limited percolation game on $\ZZ^2$, and prove the following.

\begin{proposition}
\label{prop:lmtdBwin}
Let $m \ge 1$ and $c \geq 0$ be integers.
Then Breaker can guarantee to win the $c$-boosted $(m,2m)$ limited percolation game on $\ZZ^2$ within the first $(2c+4)(2c+5)\p[big]{\ceil{ \frac{2c+2}{m} } + 1}$ rounds of the game.
\end{proposition}

\Cref{thm:boost} then follows by combining \Cref{prop:lmtd-unlmtd,prop:lmtdBwin}.

We start by providing various definitions and assumptions that we need in our proof.
Firstly, without loss of generality, we may assume that Maker's graph is not only connected after each turn of hers, but also that she is adding edges to it, one by one, so that her graph, which we denote by $C$, is connected at any single point during her turn.
We then let $C(\ell)$ be her graph after Maker claimed $\ell$ edges in total, and we denote by $C_{k} \defined C(\sum_{i=1}^{k} m_i)$ Maker's graph after she played $k$ full turns.

Recall that we denote by $\boundary C(\ell)$ the edge boundary of $C(\ell)$, and note that we include in this boundary also the edges that Breaker has already claimed.
The set of those edges in the boundary of $C(\ell)$ that are yet unclaimed by Breaker at this point of the game is denoted by $\boundary_F C(\ell)$, which stands intuitively for the `free' boundary of $C(\ell)$.
While this definition may be ambiguous as $\boundary_F C(\ell)$ changes during the turn of Breaker, we will always make it clear to which particular point we refer to when using this notation.

\begin{definition}
\label{def:edgetypes}
    For any $\ell \ge 1$, we call an edge $e \in \boundary C(\ell)$:
    \begin{itemize}
        \item \emph{Awful} if $\card{\boundary(C(\ell)+e)} - \card{\boundary C(\ell)} \leq 1$.
        \item \emph{Bad} if it is not awful, but by claiming $e$, Maker creates at least one new awful edge $f$ in $\boundary(C(\ell) + e)$ such that $f$ touches $e$.
        \item \emph{Good} otherwise.
    \end{itemize}
\end{definition}

We now provide Breaker with a strategy.

\begin{strategy}[Breaker's strategy for the $(m,b)$ limited percolation game on $\ZZ^2$]
\label{str:B-bstd-lmtd}
    Let $k \ge 1$ and assume that by the end of her $k$-th turn Maker has claimed $\ell$ edges so far in the game for some $\ell \ge 0$.
    In his $k$-th turn, Breaker claims edges one by one from $\boundary_F(C(\ell))$ in the following order of priority:
    \begin{enumerate}
        \item good edges,
        \item bad edges,
        \item awful edges.
    \end{enumerate}
\end{strategy}

Note that if at some point Breaker cannot follow \Cref{str:B-bstd-lmtd}, then it means that there are no unclaimed boundary edges in Maker's graph, which in particular means that Breaker has won the game.

We make the following straightforward observation that does not require proof.

\begin{observation}\label{obs:awfulbad}
An awful edge stays awful until it is claimed by either player.
A bad edge either stays bad or becomes awful until it is claimed by either player.
\end{observation}

Further, let
\begin{equation*}
    v_k \defined 2\card{E(C_k)} + 4 - \card{\boundary C_k},
\end{equation*}
and let $w_k$ be the number of awful edges in $\boundary_F C_k$ after $k$ turns of both Maker and of Breaker, and set $v_0 = w_0 = 0$.
Let $\boundary_G C_k \subseteq \boundary_F C_k$ denote the subset of good edges.

We consider the lexicographic order on ordered pairs $\set{ (x;y) : x,y \in \RR }$.
That is, we have
\begin{equation*}
    (x;y) > (z;w) \leftrightarrow \set{ x > z } \text{ or } \set{ x = z \text{ and } y > w }.
\end{equation*}
When we write inequalities between ordered pairs, we always refer to this ordering.

\Cref{thm:boost} is implied by the following observation and proposition.

\begin{observation}
\label{obs:vkwk}
    Let $k \geq 1$ and assume that after $k$ turns of both Maker and Breaker, Breaker has not yet won.
    Then we have $0 \leq v_k \leq 2c+3$ and $0 \leq w_k \leq 2c+4$.
\end{observation}
\begin{proof}
Firstly, recall that an awful edge is in particular an unclaimed edge, so by the definition of $w_k$ and by \Cref{lem:perimetric} we have
\begin{equation*}
    0 \le w_k \le \card{\boundary_F C_k} = \card{\boundary C_k} - 2mk \le 2c + 4.
\end{equation*}

Secondly, note that $v_k \geq 0$ simply by \Cref{lem:perimetric}.
Since Breaker has not won the game yet by the end of his $k$-th turn, and as he claims only edges from the boundary of Maker's graph, we have $\card{\boundary C_k} \ge 2mk+1$, and thus
\begin{equation*}
    v_k = 2\card{E(C_k)} + 4 - \card{\boundary C_k} \le 2c+3. \qedhere
\end{equation*}
\end{proof}

\begin{proposition}
\label{prop:vkwk}
    Let $k \ge 1$ and $c' = \ceil{\frac{2c+2}{m}}$.
    Assume that after $k+c'$ turns of both Maker and Breaker, Breaker has not yet won the game.
    Then there exists some $1 \leq r \leq c'+1$, such that $(v_{k+r};w_{k+r}) > (v_k;w_k)$.
\end{proposition}

Before proving \Cref{prop:vkwk}, let us see why it is useful for us.

\begin{claim}
\Cref{prop:lmtdBwin} is implied by \Cref{obs:vkwk} and \Cref{prop:vkwk}.
\end{claim}

\begin{proof}
    Consider the set
    \begin{equation*}
        S \defined \set[\big]{ (x;y) \st x \in  \set{0,1,2,\dotsc,2c+3} ,\, y \in \set{0,1,2,\dotsc,2c+4} }.
    \end{equation*}
    We clearly have $\card{S} = (2c+4)(2c+5)$.
    By \Cref{obs:vkwk} we have $(v_k;w_k) \in S$ for every $k$ such that Breaker has not yet won after $k$ rounds of the game.
    By \Cref{prop:vkwk} we can find strictly increasing sequence
    \begin{equation*}
        (v_{i_1};w_{i_1}) < (v_{i_2};w_{i_2}) < \dotsb
    \end{equation*}
    with $i_1=1$ and $i_{t+1}-i_t \leq c'+1$ for each $t \geq 1$.
    The result follows.
\end{proof}

It is left to prove \Cref{prop:vkwk}. We start by proving two easy lemmas.

\begin{lemma}\label{lem:vk}
    Let $k \geq 0$ and assume that after $k$ rounds of the game Breaker has not won yet.
    Assume that out of $m_{k+1}$ edges that Maker claimed in her $(k+1)$-th turn, $t$ were awful at the time they were claimed, for some $0 \le t \le m_{k+1}$.
    Then
    \begin{equation*}
        \card{\boundary C_{k+1}} - \card{\boundary C_k} \leq 2 m_{k+1} -t.
    \end{equation*}
    In particular we get
    \begin{equation*}
        v_{k+1} \ge v_k + t.
    \end{equation*}
\end{lemma}
\begin{proof}
When claimed by Maker, an edge is removed from the edge boundary of the new component, and at most three new edges are added to it, which means that $\card{\boundary C(\ell+1)} - \card{\boundary C(\ell)} \leq 2$, for any $\ell \ge 1$.
If the $\ell$-th edge claimed by Maker is awful, then by definition we have $\card{\boundary C(\ell+1)} - \card{\boundary C(\ell)} \leq 1$.
The first result follows by combining these two observations.

Then in particular, by the definition of $v_k$ and by the above, we have
\begin{align*}
    v_{k+1} &= 2\card{E(C_{k+1})} + 4 - \card{\boundary C_{k+1}} \\
    &\ge 2 \p[\big]{\card{E(C_k)} + m_{k+1} } + 4 - \card{\boundary C_k} - 2m_{k+1} + t = v_k + t. \qedhere
\end{align*}
\end{proof}

\begin{lemma}\label{lem:sum-mi}
Let $k \geq 1$ and assume that after $k$ rounds of the game, Breaker has not yet won.
Then
\begin{equation*}
    \sum_{i=1}^k m_i \geq km-1.
\end{equation*}
\end{lemma}
\begin{proof}
Assume for contradiction that $\sum_{i=1}^k m_i \leq km-2$.
Then by \Cref{lem:perimetric}, we have $\card{\boundary C_k} \le 2(km-2)+4 = 2km$.
However, within his first $k$ turns, Breaker claims precisely $2km$ edges, all in $\boundary C_k$.
Thus he must have won at latest after $k$ rounds, which is a desired contradiction.
\end{proof}

Next we show that we never create many good edges.

\begin{lemma}\label{lem:fewgoodedges}
For any $\ell \geq 1$, the number of good edges in $\boundary C(\ell+1)$ is at most one more than the number of good edges in $\boundary C(\ell)$, that is
\begin{equation*}
    \card{\boundary_G C(\ell+1)} - \card{\boundary_G C(\ell)} \le 1.
\end{equation*}
Moreover, $\card{\boundary_G C(1)}, \card{\boundary_G C(2)} \le 2$.
\end{lemma}
\begin{proof}
That $\card{\boundary_G C(1)}, \card{\boundary_G C(2)} \le 2$ follows by inspection, so it only remains to prove the first assertion.

Let $\ell \geq 1$.
By \Cref{obs:awfulbad}, no edge that was bad or awful in $\boundary C(\ell)$ can be good in $\boundary C(\ell+1)$.
So it is enough if we rule out the case that out of at most three edges in $\boundary C(\ell+1) \setminus \boundary C(\ell)$, two different ones would be good.

Let $e$ be the $l$th edge claimed by Maker. By symmetry, we may assume that $e$ is horizontal, i.e. that  $e= \set{ (x,y), (x+1,y) }$ for some $x,y \in \ZZ$.
Further, since Maker always claims an edge in the edge boundary of her only connected component and due to symmetry again, we only need to consider the following two cases.
\begin{enumerate}
    \item The edge $\set[\big]{(x-1,y), (x,y)}$ was already in $C(\ell)$.
    \item The edge $\set[\big]{(x,y+1), (x,y)}$ was already in $C(\ell)$.
\end{enumerate}
In either case, we have that $\boundary C(\ell+1) \setminus \boundary C(\ell)$ is contained in the set
\begin{equation*}
    \set[\Big]{ \set[\big]{(x+1,y), (x+2,y)}, \set[\big]{(x+1,y), (x+1,y-1)}, \set[\big]{(x+1,y), (x+1,y+1)} },
\end{equation*}
and out of these three edges, only $\set[\big]{(x+1,y), (x+2,y)}$ may be in $\boundary_G C(\ell+1)$.
\end{proof}

\Cref{lem:fewgoodedges} has the following corollary.

\begin{corollary}\label{cor:fewgoodedges}
    For any $k \ge 1$, at the end of Breaker's $k$-th turn, we have
    \begin{equation*}
        \card{\boundary_G C_k} \le (c - mk)_+,
    \end{equation*}
    where $n_+ \defined \max \set{n, 0}$.
\end{corollary}

\begin{proof}
If $m \ge 3$, then we can assume $\card{E(C_1)} \ge 2$, otherwise Breaker wins already in the first round.
And in the case when $\card{E(C_1)} \ge 2$, by \Cref{lem:fewgoodedges}, we get $\card{\boundary_G C_1} \le \card{E(C_1)}$ at the end of Maker's first turn.
In his $k$-th turn, Breaker claims at least $\min \set{ 2m, \card{\boundary_G C_k} }$ good edges.
As this holds for any $k \ge 1$, at the end of Breaker's $k$-th turn we have
\begin{equation*}
    \card{\boundary_G C_k} \le \p[\big]{\card{E(C_k)} - 2mk }_+ \le (c - mk)_+.
\end{equation*}

If $m \in \set{1,2}$ and $m_1 \le 1$, then although we can have $\card{\boundary_G C_1} > \card{E(C_1)} $ at the end of Maker's first turn, we still have $\card{\boundary_G C_1} = 0$ by the end of Breaker's turn, and the result follows similarly to the first case.
\end{proof}

We are now ready to prove \Cref{prop:vkwk}.

\begin{proof}[Proof of \Cref{prop:vkwk}]
The crucial part of our proof is the following claim.

\begin{claim}\label{clm:crucial}
There exists $1 \le j \le c'-1$, such that in his $(k+j)$-th turn Breaker claimed a bad or an awful edge.
\end{claim}
\begin{proof}[Proof of \Cref{clm:crucial}]
Assume for contradiction that in his $(k+j)$-th turn Breaker claimed only good edges for all $1 \le j \le c'-1$, meaning he claimed a total of $2m(c'-1)$ good edges during these rounds.

By \Cref{cor:fewgoodedges}, at the end of Breaker's $k$-th turn we have $\card{\boundary_G C_k} \le (c - mk)_+$.
Moreover, by \Cref{lem:sum-mi} Maker claimed at most
\begin{equation*}
    \sum_{i=k+1}^{k+c'-1} m_i \le c + (k+c'-1)m - km + 1 = c + 1 + m(c'-1)
\end{equation*}
edges after her $k$-th turn.
Hence, by \Cref{lem:fewgoodedges} at most $c + 1 + m(c'-1)$ new good edges were added to the boundary of Maker's graph in rounds $k+1, \dotsc, k+c'-1$. And by \Cref{obs:awfulbad} no awful or bad edge became good, so we get
\begin{equation*}
    2m(c'-1) \le (c - mk)_+ + c + 1 + m(c'-1),
\end{equation*}
contradicting $c' = \ceil{\frac{2c+2}{m}}$.
\end{proof}

Recall that by \Cref{lem:vk} the sequence $v_i$ is non-decreasing.
Pick the smallest such $j$ from \Cref{clm:crucial}.
In $(k+j+1)$-st turn of Maker, Maker claims some bad or awful edge.

If Maker claims an awful edge (or claimed an awful edge in any other of the rounds $k+1,\dotsc,k+j+1$), we have $v_{k+j+1} > v_k$ and we are done.

If Maker claims no awful edge (and did not in any other of the rounds $k+1,\dotsc,k+j+1$) but claims a bad edge his $(k+j+1)$-th turn, note that Breaker also must have claimed no awful edge in his $(k+j)$-th turn. Next, there are two options.
Either Breaker claims no awful edge in his $(k+j+1)$-th turn, and then $w_{k+j+1}>w_k$ and we are done (as $v_{k+j+1}=v_k$).
Or Breaker claims some awful edge in his $(k+j+1)$-th turn, but then Maker claims some awful edge in her $(k+j+2)$-th turn, hence $v_{k+j+2}>v_k$ and we are again done.

Hence, the proof of \Cref{prop:vkwk} is finished.
\end{proof}

This now also concludes the proof of \Cref{thm:boost}.


\section{Polluted board}
\label{sec:polluted}

In this section, we consider the Maker-Breaker percolation game on a random board, as suggested by Day and Falgas-Ravry~\cite{day2020maker2}.
As we considered so far the setting of $(\Lambda, v_0) = (\ZZ^2, (0,0))$, it is natural to consider the same game after bond percolation is performed.

Consider an infinite connected graph $\Lambda = (V,E)$ and $p \in [0,1]$.
In \indef{bond percolation}, each edge of $\Lambda$ is declared \indef{open} with probability $p$, independently from all the other edges.
An edge is said to be \indef{closed} if it is not open.
We denote by $(\Lambda)_p$ the random subgraph of $\Lambda$ with vertex set $V$ and edge set $\set{ e \in E \st e \text{ is open} }$.
Denote by $\Pperc{p}$ the probability measure induced by this process.

The most striking property of this probabilistic model when $\Lambda = \ZZ^2$ is that it undergoes a phase transition at $p = 1/2$.
Indeed, if by an \indef{open cluster} we mean a connected component of open edges, then $\pperc{p}{\text{exists infinite open cluster}} = 0$ for all $p \leq 1/2$, as shown by Harris~\cite{harris1960lower}, and $\pperc{p}{\text{exists infinite open cluster}} = 1$ for all $p > 1/2$, as shown by Kesten~\cite{kesten1980upper}.
For simpler proofs, see Bollobás and Riordan~\cites{bollobas-riordan2006short,bollobas-riordan2007note},
and for a comprehensive introduction to the subject, see the book of Bollobás and Riordan~\cite{bollobas2006percolation}.

In this section, we are concerned with the Maker-Breaker percolation game on $(\ZZ^2)_p$, so to avoid confusion, we refer to the board after performing bond percolation as a \indef{polluted board}.

If $p \leq 1/2$, then a.s. (almost surely) the open cluster that contains the origin is finite, so Breaker a.s. wins trivially.
On the other hand, if $p > 1/2$, then a.s. there is an infinite open cluster.
There is no guarantee however that the origin will be part of this infinite cluster.
To make the game more interesting, as suggested by Day and Falgas-Ravry~\cite{day2020maker2}, we allow Maker to choose the initial vertex $v_0$ after the randomness of the bond percolation is realised.

\begin{definition}
\label{def:polluted}
    We define the $(m, b)$ Maker-Breaker percolation game on a $p$-polluted board $\Lambda$ as follows.
    Before the game starts, bond percolation with probability parameter $p$ is performed in the graph $\Lambda$ to obtain the polluted board $(\Lambda)_p$.
    After all the randomness was realised, Maker chooses a vertex $v_0$.
    From now onward, the game proceeds as the usual $(m,b)$ Maker-Breaker percolation game on $((\Lambda)_p, v_0)$ that we defined previously.
\end{definition}

We abbreviate as before, so we refer to the game above as the \indef{$(m,b)$-game on $p$-polluted $\Lambda$}.
Note that all the randomness was realised before the game started, so it is a deterministic game played on a random board.
We are interested in strategies that win almost surely on the choice of the board.

The main result of this section is \Cref{thm:polluted}, concerning the $(1,1)$-game on $p$-polluted $\ZZ^2$, so from now on, unless stated otherwise, we are considering this game.

To prove \Cref{thm:polluted}, we present a winning strategy for Breaker that heavily relies on the closely related model of the north-east oriented percolation on $\ZZ^2$, which we will now briefly describe.

Orient each edge of $\ZZ^2$ in the direction of the increasing coordinate value.
That is, each horizontal edge is oriented to the east and each vertical edges is oriented to the north.
Next, each oriented edge is said to be \indef{open} with probability $p$,
independently of all the other edges.
An oriented edge is said to be \indef{closed} if it is not open.
We refer to the probability measure $\Porient{p}$ we obtain as the \indef{north-east oriented percolation on $\ZZ^2$}.
For a detailed overview of the results in this model, see the survey of Durrett~\cite{durrett1984oriented}.

In the bond percolation, we are interested in whether there is or there is not an infinite open cluster. In the north-east oriented percolation, we are interested whether there exists an infinite oriented open path.
Define the critical probability
\begin{equation*}
    p^{*} \defined \inf \set*{ p \in [0,1] \st
    \porient*{p}{(0,0) \text{ is on an infinite open oriented path}} > 0 }.
\end{equation*}
It is known that $0.6298 \leq p^{*} \leq 0.6735$, where the lower bound is due to Dhar~\cite{dhar1982directed}, improving on the ideas of Gray, Wierman and Smythe~\cite{gray1980lower}, and the upper bound is due to Balister, Bollobás and Stacey~\cite{balister1994improved}.
Like the threshold in the bond percolation, there is an infinite oriented open path a.s. if $p > p^{*}$ and there is no infinite oriented open path a.s. if $p \leq p^{*}$ in the north-east oriented percolation.

This model is relevant to us via a simple coupling argument.
Note that by forgetting the orientation of the edges, we obtain that for $p > p^{*}$,
\begin{align*}
    &\pperc*{p}{\text{there is an infinite up-right open path from } 0} \\
    &= \porient*{p}{(0,0) \text{ is a starting point of infinite oriented open path}} > 0.
\end{align*}

We say that a subgraph of $\ZZ^2$ is \emph{barred} if there is no infinite path going only in two directions.
That is, if for all vertices $v \in \ZZ^2$, there is no north-east, no north-west, no south-east and no south-west infinite path starting from $v$.
Coupling the bond percolation with four rotated copies of the north-east oriented percolation model as above gives us the following result.

\begin{proposition}
\label{prop:barred}
If $p < p^{*}$, then a.s. each open cluster of $(\ZZ^2)_p$ is barred.
\end{proposition}

We will prove the following proposition, which together with \Cref{prop:barred} implies \Cref{thm:polluted}.

\begin{proposition}
\label{prop:barred-strategy}
    Fix any barred percolation configuration of $\ZZ^2$.
    Then regardless of which point is fixed by Maker as the origin, Breaker has a winning strategy for the $(1,1)$ Maker-Breaker percolation game on this board.
\end{proposition}

We say that a path is \emph{unbarred} if it only ever uses at most two directions.
Therefore, a barred configuration in $\ZZ^2$ is precisely a configuration with no infinite unbarred paths.

For $d \in \NN$ and $v \in \ZZ^2$, let $B_d(v)= v + [-d,d]^2$ be the ball of radius $d$, centred at $v$ in the $\ell_\infty$-norm.
We say that an edge is contained in $B_d(v)$ if both of its endpoints are.
We also write $B_d$ for the same ball centred at $(0,0)$.
Using a very simple argument in the style of an infinite Ramsey theory, we prove the following lemma.

\begin{lemma}
\label{lem:inframsey}
    Consider any fixed barred percolation configuration of $\ZZ^2$.
    Then for every $v \in \ZZ^2$, there exists $d = d(v) \in \NN$
    such that any unbarred path starting from $v$ is contained within $B_{d}(v)$.
\end{lemma}
\begin{proof}
The subtle case we need to treat here is the possibility that there are infinitely many finite unbarred paths from $v$, but still, no infinite unbarred path.

Let $S$ be the set of the vertices of $\ZZ^2$ which are reachable from $v$ using unbarred paths.
If $S$ is finite, then the result follows.
Now we assume $S$ is infinite and we go for a contradiction.

Without loss of generality, the set $S_0 \subseteq S$ of the vertices reachable from $v$ via north-east paths is infinite.
So we can find a vertex $v_1$, such that: $v_1$ is a neighbour of $v$, either at north or east of $v$; the edge $\set{v,v_1}$ is open in the configuration; and such that the set $S_1 \subseteq S_0$ of the vertices reachable by north-east paths from $v_1$ is infinite.

Continuing inductively, we find a sequence of vertices $v_0, v_1, v_2 \dotsc$ such that $v_{i+1}$ is a neighbour of $v_i$ at north or east, the edge $\set{v_i, v_{i+1}}$ is open and $v_0 = v$.
But this is precisely an unbarred infinite path from $v$, which contradicts the configuration being barred.
\end{proof}

We now define a strategy for Breaker that wins on any barred board, and any choice of $v \in \ZZ^2$ as an origin.
In fact, since a translated board is still barred, we assume without loss of generality that $v = (0,0)$.
By \Cref{lem:inframsey}, there is an integer $d$ such that any unbarred path from $(0,0)$ is contained in $B_d$.
Breaker's strategy exploits the existence of this box.

A vertex in $\ZZ^2$ is called \indef{axial} if one of its coordinates coincides with a coordinate of $v$ (i.e. if one of its coordinates is $0$, as we took $v=(0,0)$).
An edge is called \indef{non-axial} if one of its endpoints is not axial.
Breaker's strategy relies on the following pairing.

\begin{definition}
\label{def:pairing}
The \indef{barrier pairing}
is the following pairing of the of the non-axial edges of $\ZZ^2$.
Let $(x,y) \in \ZZ^2$ be a non-axial vertex.
\begin{enumerate}
    \item If $x > 0$ and $y > 0$,
        then pair $\set{(x,y), (x,y-1)}$ with $\set{(x,y),(x-1,y)}$;
    \item If $x > 0$ and $y < 0$,
        then pair $\set{(x,y), (x,y+1)}$ with $\set{(x,y),(x-1,y)}$;
    \item If $x < 0$ and $y > 0$,
        then pair $\set{(x,y), (x,y-1)}$ with $\set{(x,y),(x+1,y)}$;
    \item If $x < 0$ and $y < 0$,
        then pair $\set{(x,y), (x,y+1)}$ with $\set{(x,y),(x+1,y)}$.
\end{enumerate}
\end{definition}

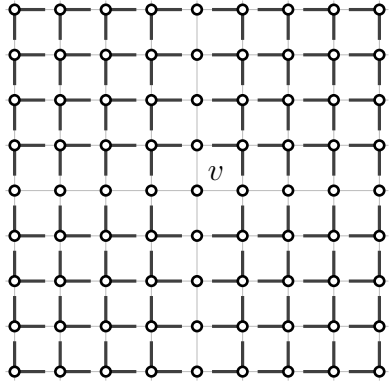
\begin{figure}[ht!]
\centering
\begin{tikzpicture}[scale=0.6]
    \tikzstyle{s-helper}=[lightgray, thin];
    \tikzstyle{s-bold}=[darkgray, very thick];
    \draw[step=1, s-helper] (-4.2,-4.2) grid (4.2, 4.2);
    \foreach \x in {0,...,3}
    \foreach \y in {0,...,3}
    {
        \draw[xshift= \x cm, yshift= \y cm, s-bold]
            (1, 0.33) -- (1, 1) -- (0.33,1);
        \draw[xshift=-\x cm, yshift= \y cm, s-bold]
            (-1, 0.33) -- (-1, 1) -- (-0.33,1);
        \draw[xshift= \x cm, yshift=-\y cm, s-bold]
            (1, -0.33) -- (1, -1) -- (0.33,-1);
        \draw[xshift=-\x cm, yshift=-\y cm, s-bold]
            (-1, -0.33) -- (-1, -1) -- (-0.33,-1);
    }
    \foreach \x in {-4,...,4}
    \foreach \y in {-4,...,4}
        \filldraw[xshift=\x cm, yshift=\y cm, black, fill=white, line width=1.0pt] (0,0) circle (3pt);
    \draw (0,0) node[anchor=south west] {$v$};
\end{tikzpicture}
\caption{The barrier pairing around $v$}
\label{fig:pairing}
\end{figure}

We now define the following strategy for Breaker for the $(1,1)$-game on the $p$-polluted square lattice $\ZZ^2$.

\begin{strategy}[Breaker's strategy for the $(1,1)$-game on a barred board]
\label{str:barred}
Consider a barred configuration of edges of $\ZZ^2$ and a fixed origin at $(0,0)$.
If the open cluster of the origin is finite, then Breaker wins immediately.
Let $d$ be minimum such that any unbarred path starting from the origin is contained within $B_d$, which exists by \Cref{lem:inframsey}.
Breaker will react to each move of Maker using the barrier pairing in \Cref{def:pairing} as follows.
\begin{enumerate}
    \item If Maker claims a non-axial edge inside $B_{d+1}$, and its paired edge is both open and unclaimed, then Breaker claims that edge;
    \item Otherwise, if there is any open unclaimed edge inside $B_{d+1}$, breaker claims one of them arbitrarily;
    \item Otherwise, Breaker claims an arbitrary open unclaimed edge in the open cluster of the origin.
\end{enumerate}
\end{strategy}

\begin{figure}[ht!]
\centering
\begin{tikzpicture}[scale=0.4]
    \tikzstyle{s-helper}=[lightgray, thin];
    \tikzstyle{s-bold}=[darkgray, very thick];
    \definecolor{s-col}{RGB}{253,174,97};
    \tikzstyle{s-hi}=[ultra thick, s-col];
    \draw[step=1, s-helper] (-15.2,-15.2) grid (15.2, 15.2);
    \draw[thick]
    (-14,15) -- (-14,14)    (-13,15) -- (-13,14)
    (-12,15) -- (-12,14)    (-11,15) -- (-11,14)
    (-10,15) -- (-10,14)    (-8,15) -- (-8,14)
    (-6,15) -- (-6,14)    (-2,15) -- (-2,14)
    (-1,15) -- (-1,14)    (4,15) -- (4,14)
    (6,15) -- (6,14)    (8,15) -- (8,14)
    (9,15) -- (9,14)    (15,15) -- (15,14)
    (-15,14) -- (-15,13)    (-14,14) -- (-14,13)
    (-13,14) -- (-13,13)    (-9,14) -- (-9,13)
    (-8,14) -- (-8,13)    (-5,14) -- (-5,13)
    (-2,14) -- (-2,13)    (-1,14) -- (-1,13)
    (1,14) -- (1,13)    (4,14) -- (4,13)
    (6,14) -- (6,13)    (7,14) -- (7,13)
    (9,14) -- (9,13)    (11,14) -- (11,13)
    (12,14) -- (12,13)    (13,14) -- (13,13)
    (15,14) -- (15,13)
    (-15,13) -- (-15,12)    (-13,13) -- (-13,12)
    (-12,13) -- (-12,12)    (-9,13) -- (-9,12)
    (-8,13) -- (-8,12)    (-7,13) -- (-7,12)
    (-6,13) -- (-6,12)    (-2,13) -- (-2,12)
    (-1,13) -- (-1,12)    (1,13) -- (1,12)
    (2,13) -- (2,12)    (4,13) -- (4,12)
    (6,13) -- (6,12)    (8,13) -- (8,12)
    (14,13) -- (14,12)    (15,13) -- (15,12)
    (-13,12) -- (-13,11)    (-12,12) -- (-12,11)
    (-10,12) -- (-10,11)    (-8,12) -- (-8,11)
    (-5,12) -- (-5,11)    (-1,12) -- (-1,11)
    (0,12) -- (0,11)    (2,12) -- (2,11)
    (3,12) -- (3,11)    (4,12) -- (4,11)
    (5,12) -- (5,11)    (7,12) -- (7,11)
    (8,12) -- (8,11)    (12,12) -- (12,11)
    (13,12) -- (13,11)    (14,12) -- (14,11)
    (15,12) -- (15,11)
    (-14,11) -- (-14,10)    (-12,11) -- (-12,10)
    (-11,11) -- (-11,10)    (-10,11) -- (-10,10)
    (-5,11) -- (-5,10)    (-1,11) -- (-1,10)
    (0,11) -- (0,10)    (8,11) -- (8,10)
    (9,11) -- (9,10)    (13,11) -- (13,10)
    (-15,10) -- (-15,9)    (-14,10) -- (-14,9)
    (-12,10) -- (-12,9)    (-8,10) -- (-8,9)
    (-4,10) -- (-4,9)    (0,10) -- (0,9)
    (1,10) -- (1,9)    (2,10) -- (2,9)
    (3,10) -- (3,9)    (4,10) -- (4,9)
    (7,10) -- (7,9)    (9,10) -- (9,9)
    (11,10) -- (11,9)    (13,10) -- (13,9)
    (14,10) -- (14,9)    (15,10) -- (15,9)
    (-15,9) -- (-15,8)    (-14,9) -- (-14,8)
    (-12,9) -- (-12,8)    (-10,9) -- (-10,8)
    (-8,9) -- (-8,8)    (-2,9) -- (-2,8)
    (0,9) -- (0,8)    (1,9) -- (1,8)
    (8,9) -- (8,8)    (11,9) -- (11,8)
    (12,9) -- (12,8)    (13,9) -- (13,8)
    (14,9) -- (14,8)    (15,9) -- (15,8)
    (-15,8) -- (-15,7)    (-11,8) -- (-11,7)
    (-9,8) -- (-9,7)    (-8,8) -- (-8,7)
    (0,8) -- (0,7)    (1,8) -- (1,7)
    (7,8) -- (7,7)    (11,8) -- (11,7)
    (12,8) -- (12,7)    (13,8) -- (13,7)
    (14,8) -- (14,7)    (15,8) -- (15,7)
    (-15,7) -- (-15,6)    (-13,7) -- (-13,6)
    (-12,7) -- (-12,6)    (-11,7) -- (-11,6)
    (-6,7) -- (-6,6)    (0,7) -- (0,6)
    (13,7) -- (13,6)    (14,7) -- (14,6)
    (-15,6) -- (-15,5)    (-14,6) -- (-14,5)
    (-11,6) -- (-11,5)    (-9,6) -- (-9,5)
    (-8,6) -- (-8,5)    (-6,6) -- (-6,5)
    (-4,6) -- (-4,5)    (0,6) -- (0,5)
    (3,6) -- (3,5)    (11,6) -- (11,5)
    (14,6) -- (14,5)    (15,6) -- (15,5)
    (-15,5) -- (-15,4)    (-13,5) -- (-13,4)
    (-12,5) -- (-12,4)    (-11,5) -- (-11,4)
    (-9,5) -- (-9,4)    (-8,5) -- (-8,4)
    (-7,5) -- (-7,4)    (-6,5) -- (-6,4)
    (-3,5) -- (-3,4)    (-2,5) -- (-2,4)
    (0,5) -- (0,4)    (1,5) -- (1,4)
    (6,5) -- (6,4)    (8,5) -- (8,4)
    (9,5) -- (9,4)    (10,5) -- (10,4)
    (11,5) -- (11,4)    (12,5) -- (12,4)
    (13,5) -- (13,4)    (14,5) -- (14,4)
    (-15,4) -- (-15,3)    (-13,4) -- (-13,3)
    (-12,4) -- (-12,3)    (-9,4) -- (-9,3)
    (-7,4) -- (-7,3)    (-3,4) -- (-3,3)
    (6,4) -- (6,3)    (7,4) -- (7,3)
    (8,4) -- (8,3)    (12,4) -- (12,3)
    (14,4) -- (14,3)    (15,4) -- (15,3)
    (-15,3) -- (-15,2)    (-12,3) -- (-12,2)
    (-10,3) -- (-10,2)    (-9,3) -- (-9,2)
    (-8,3) -- (-8,2)    (-7,3) -- (-7,2)
    (-4,3) -- (-4,2)    (3,3) -- (3,2)
    (4,3) -- (4,2)    (5,3) -- (5,2)
    (6,3) -- (6,2)    (8,3) -- (8,2)
    (10,3) -- (10,2)    (12,3) -- (12,2)
    (13,3) -- (13,2)
    (-12,2) -- (-12,1)    (-11,2) -- (-11,1)
    (-10,2) -- (-10,1)    (-9,2) -- (-9,1)
    (-8,2) -- (-8,1)    (-5,2) -- (-5,1)
    (-2,2) -- (-2,1)    (5,2) -- (5,1)
    (7,2) -- (7,1)    (10,2) -- (10,1)
    (13,2) -- (13,1)    (15,2) -- (15,1)
    (-15,1) -- (-15,0)    (-14,1) -- (-14,0)
    (-13,1) -- (-13,0)    (-12,1) -- (-12,0)
    (-10,1) -- (-10,0)    (-9,1) -- (-9,0)
    (-8,1) -- (-8,0)    (-7,1) -- (-7,0)
    (-5,1) -- (-5,0)    (3,1) -- (3,0)
    (4,1) -- (4,0)    (6,1) -- (6,0)
    (9,1) -- (9,0)    (10,1) -- (10,0)
    (11,1) -- (11,0)    (13,1) -- (13,0)
    (14,1) -- (14,0)    (15,1) -- (15,0)
    (-15,0) -- (-15,-1)    (-14,0) -- (-14,-1)
    (-12,0) -- (-12,-1)    (-11,0) -- (-11,-1)
    (-10,0) -- (-10,-1)    (-9,0) -- (-9,-1)
    (-8,0) -- (-8,-1)    (-7,0) -- (-7,-1)
    (-6,0) -- (-6,-1)    (-5,0) -- (-5,-1)
    (3,0) -- (3,-1)    (4,0) -- (4,-1)
    (5,0) -- (5,-1)    (6,0) -- (6,-1)
    (7,0) -- (7,-1)    (8,0) -- (8,-1)
    (10,0) -- (10,-1)    (13,0) -- (13,-1)
    (14,0) -- (14,-1)
    (-15,-1) -- (-15,-2)    (-14,-1) -- (-14,-2)
    (-11,-1) -- (-11,-2)    (-10,-1) -- (-10,-2)
    (-9,-1) -- (-9,-2)    (-6,-1) -- (-6,-2)
    (-3,-1) -- (-3,-2)    (5,-1) -- (5,-2)
    (7,-1) -- (7,-2)    (12,-1) -- (12,-2)
    (13,-1) -- (13,-2)    (14,-1) -- (14,-2)
    (-15,-2) -- (-15,-3)    (-12,-2) -- (-12,-3)
    (-11,-2) -- (-11,-3)    (-10,-2) -- (-10,-3)
    (-8,-2) -- (-8,-3)    (-4,-2) -- (-4,-3)
    (4,-2) -- (4,-3)    (6,-2) -- (6,-3)
    (7,-2) -- (7,-3)    (9,-2) -- (9,-3)
    (10,-2) -- (10,-3)    (11,-2) -- (11,-3)
    (12,-2) -- (12,-3)
    (-13,-3) -- (-13,-4)    (-12,-3) -- (-12,-4)
    (-10,-3) -- (-10,-4)    (-9,-3) -- (-9,-4)
    (-5,-3) -- (-5,-4)    (-4,-3) -- (-4,-4)
    (-3,-3) -- (-3,-4)    (3,-3) -- (3,-4)
    (5,-3) -- (5,-4)    (7,-3) -- (7,-4)
    (9,-3) -- (9,-4)    (12,-3) -- (12,-4)
    (13,-3) -- (13,-4)    (15,-3) -- (15,-4)
    (-15,-4) -- (-15,-5)    (-14,-4) -- (-14,-5)
    (-9,-4) -- (-9,-5)    (-6,-4) -- (-6,-5)
    (-5,-4) -- (-5,-5)    (-3,-4) -- (-3,-5)
    (-2,-4) -- (-2,-5)    (-1,-4) -- (-1,-5)
    (3,-4) -- (3,-5)    (6,-4) -- (6,-5)
    (8,-4) -- (8,-5)    (9,-4) -- (9,-5)
    (10,-4) -- (10,-5)    (13,-4) -- (13,-5)
    (14,-4) -- (14,-5)
    (-11,-5) -- (-11,-6)    (-9,-5) -- (-9,-6)
    (-8,-5) -- (-8,-6)    (-7,-5) -- (-7,-6)
    (1,-5) -- (1,-6)    (2,-5) -- (2,-6)
    (5,-5) -- (5,-6)    (6,-5) -- (6,-6)
    (7,-5) -- (7,-6)    (9,-5) -- (9,-6)
    (12,-5) -- (12,-6)    (15,-5) -- (15,-6)
    (-14,-6) -- (-14,-7)    (-12,-6) -- (-12,-7)
    (-11,-6) -- (-11,-7)    (-10,-6) -- (-10,-7)
    (-6,-6) -- (-6,-7)    (1,-6) -- (1,-7)
    (3,-6) -- (3,-7)    (4,-6) -- (4,-7)
    (5,-6) -- (5,-7)    (9,-6) -- (9,-7)
    (11,-6) -- (11,-7)    (15,-6) -- (15,-7)
    (-14,-7) -- (-14,-8)    (-11,-7) -- (-11,-8)
    (-4,-7) -- (-4,-8)    (2,-7) -- (2,-8)
    (4,-7) -- (4,-8)    (5,-7) -- (5,-8)
    (7,-7) -- (7,-8)    (8,-7) -- (8,-8)
    (10,-7) -- (10,-8)    (11,-7) -- (11,-8)
    (13,-7) -- (13,-8)    (15,-7) -- (15,-8)
    (-15,-8) -- (-15,-9)    (-14,-8) -- (-14,-9)
    (-13,-8) -- (-13,-9)    (-9,-8) -- (-9,-9)
    (-8,-8) -- (-8,-9)    (-7,-8) -- (-7,-9)
    (-4,-8) -- (-4,-9)    (-3,-8) -- (-3,-9)
    (-1,-8) -- (-1,-9)    (1,-8) -- (1,-9)
    (3,-8) -- (3,-9)    (5,-8) -- (5,-9)
    (7,-8) -- (7,-9)    (8,-8) -- (8,-9)
    (10,-8) -- (10,-9)    (12,-8) -- (12,-9)
    (13,-8) -- (13,-9)    (14,-8) -- (14,-9)
    (-15,-9) -- (-15,-10)    (-14,-9) -- (-14,-10)
    (-13,-9) -- (-13,-10)    (-10,-9) -- (-10,-10)
    (-4,-9) -- (-4,-10)    (-1,-9) -- (-1,-10)
    (0,-9) -- (0,-10)    (1,-9) -- (1,-10)
    (2,-9) -- (2,-10)    (4,-9) -- (4,-10)
    (5,-9) -- (5,-10)    (7,-9) -- (7,-10)
    (11,-9) -- (11,-10)    (13,-9) -- (13,-10)
    (14,-9) -- (14,-10)    (15,-9) -- (15,-10)
    (-12,-10) -- (-12,-11)    (-11,-10) -- (-11,-11)
    (-10,-10) -- (-10,-11)    (-9,-10) -- (-9,-11)
    (-7,-10) -- (-7,-11)    (-6,-10) -- (-6,-11)
    (-5,-10) -- (-5,-11)    (1,-10) -- (1,-11)
    (2,-10) -- (2,-11)    (4,-10) -- (4,-11)
    (5,-10) -- (5,-11)    (7,-10) -- (7,-11)
    (8,-10) -- (8,-11)    (9,-10) -- (9,-11)
    (10,-10) -- (10,-11)    (11,-10) -- (11,-11)
    (12,-10) -- (12,-11)    (15,-10) -- (15,-11)
    (-15,-11) -- (-15,-12)    (-14,-11) -- (-14,-12)
    (-13,-11) -- (-13,-12)    (-12,-11) -- (-12,-12)
    (-10,-11) -- (-10,-12)    (-9,-11) -- (-9,-12)
    (-7,-11) -- (-7,-12)    (-6,-11) -- (-6,-12)
    (-5,-11) -- (-5,-12)    (-2,-11) -- (-2,-12)
    (1,-11) -- (1,-12)    (2,-11) -- (2,-12)
    (3,-11) -- (3,-12)    (5,-11) -- (5,-12)
    (8,-11) -- (8,-12)    (9,-11) -- (9,-12)
    (10,-11) -- (10,-12)    (14,-11) -- (14,-12)
    (15,-11) -- (15,-12)
    (-14,-12) -- (-14,-13)    (-13,-12) -- (-13,-13)
    (-11,-12) -- (-11,-13)    (-9,-12) -- (-9,-13)
    (-1,-12) -- (-1,-13)    (0,-12) -- (0,-13)
    (1,-12) -- (1,-13)    (3,-12) -- (3,-13)
    (4,-12) -- (4,-13)    (10,-12) -- (10,-13)
    (11,-12) -- (11,-13)    (12,-12) -- (12,-13)
    (13,-12) -- (13,-13)    (14,-12) -- (14,-13)
    (15,-12) -- (15,-13)
    (-11,-13) -- (-11,-14)    (-9,-13) -- (-9,-14)
    (-8,-13) -- (-8,-14)    (-7,-13) -- (-7,-14)
    (-6,-13) -- (-6,-14)    (-5,-13) -- (-5,-14)
    (-4,-13) -- (-4,-14)    (-2,-13) -- (-2,-14)
    (0,-13) -- (0,-14)    (1,-13) -- (1,-14)
    (2,-13) -- (2,-14)    (3,-13) -- (3,-14)
    (7,-13) -- (7,-14)    (9,-13) -- (9,-14)
    (12,-13) -- (12,-14)    (13,-13) -- (13,-14)
    (14,-13) -- (14,-14)    (15,-13) -- (15,-14)
    (-15,-14) -- (-15,-15)    (-12,-14) -- (-12,-15)
    (-11,-14) -- (-11,-15)    (-10,-14) -- (-10,-15)
    (-9,-14) -- (-9,-15)    (-7,-14) -- (-7,-15)
    (-6,-14) -- (-6,-15)    (-5,-14) -- (-5,-15)
    (-4,-14) -- (-4,-15)    (-3,-14) -- (-3,-15)
    (-1,-14) -- (-1,-15)    (0,-14) -- (0,-15)
    (1,-14) -- (1,-15)    (4,-14) -- (4,-15)
    (5,-14) -- (5,-15)    (6,-14) -- (6,-15)
    (10,-14) -- (10,-15)    (11,-14) -- (11,-15)
    ;

    \draw[thick]
    (-15, 15) -- (-12, 15)    (-11, 15) -- (-10, 15)
    (-9, 15) -- (-7, 15)    (-6, 15) -- (-2, 15)
    (2, 15) -- (3, 15)    (4, 15) -- (6, 15)
    (7, 15) -- (8, 15)    (10, 15) -- (11, 15)
    (14, 15) -- (15, 15)
    (-15, 14) -- (-14, 14)    (-13, 14) -- (-12, 14)
    (-9, 14) -- (-6, 14)    (-2, 14) -- (-1, 14)
    (1, 14) -- (2, 14)    (4, 14) -- (5, 14)
    (8, 14) -- (9, 14)    (11, 14) -- (15, 14)
    (-12, 13) -- (-10, 13)    (-8, 13) -- (-6, 13)
    (-1, 13) -- (2, 13)    (4, 13) -- (6, 13)
    (8, 13) -- (9, 13)    (11, 13) -- (13, 13)
    (14, 13) -- (15, 13)
    (-14, 12) -- (-13, 12)    (-10, 12) -- (-7, 12)
    (-3, 12) -- (-2, 12)    (-1, 12) -- (0, 12)
    (1, 12) -- (2, 12)    (3, 12) -- (4, 12)
    (5, 12) -- (9, 12)    (13, 12) -- (15, 12)
    (-15, 11) -- (-14, 11)    (-12, 11) -- (-9, 11)
    (-7, 11) -- (-5, 11)    (-3, 11) -- (-1, 11)
    (1, 11) -- (2, 11)    (3, 11) -- (8, 11)
    (11, 11) -- (13, 11)    (14, 11) -- (15, 11)
    (-15, 10) -- (-11, 10)    (-9, 10) -- (-8, 10)
    (-7, 10) -- (-6, 10)    (-5, 10) -- (-4, 10)
    (1, 10) -- (4, 10)    (7, 10) -- (10, 10)
    (12, 10) -- (15, 10)
    (-15, 9) -- (-14, 9)    (-12, 9) -- (-10, 9)
    (-8, 9) -- (-4, 9)    (-1, 9) -- (1, 9)
    (3, 9) -- (4, 9)    (7, 9) -- (11, 9)
    (12, 9) -- (13, 9)    (14, 9) -- (15, 9)
    (-13, 8) -- (-10, 8)    (-3, 8) -- (-2, 8)
    (-1, 8) -- (0, 8)   (1, 8) -- (2, 8)
    (8, 8) -- (10, 8)   (11, 8) -- (15, 8)
    (-14, 7) -- (-10, 7)    (-9, 7) -- (-7, 7)
    (-1, 7) -- (0, 7)    (1, 7) -- (2, 7)
    (3, 7) -- (4, 7)    (12, 7) -- (13, 7)
    (14, 7) -- (15, 7)
    (-15, 6) -- (-14, 6)    (-11, 6) -- (-9, 6)
    (-8, 6) -- (-7, 6)    (-6, 6) -- (-5, 6)
    (0, 6) -- (2, 6)    (3, 6) -- (4, 6)
    (14, 6) -- (15, 6)
    (-15, 5) -- (-14, 5)    (-12, 5) -- (-10, 5)
    (-9, 5) -- (-5, 5)    (-4, 5) -- (-2, 5)
    (0, 5) -- (2, 5)    (10, 5) -- (12, 5)
    (-15, 4) -- (-14, 4)    (-13, 4) -- (-12, 4)
    (-11, 4) -- (-9, 4)    (-7, 4) -- (-6, 4)
    (-4, 4) -- (-2, 4)    (-1, 4) -- (1, 4)
    (6, 4) -- (8, 4)    (9, 4) -- (11, 4)
    (12, 4) -- (14, 4)
    (-14, 3) -- (-13, 3)    (-11, 3) -- (-10, 3)
    (-7, 3) -- (-5, 3)    (-4, 3) -- (-3, 3)
    (8, 3) -- (9, 3)    (10, 3) -- (11, 3)
    (12, 3) -- (14, 3)
    (-15, 2) -- (-11, 2)    (-8, 2) -- (-5, 2)
    (3, 2) -- (4, 2)    (5, 2) -- (6, 2)
    (7, 2) -- (8, 2)    (9, 2) -- (11, 2)
    (12, 2) -- (15, 2)
    (-15, 1) -- (-14, 1)    (-12, 1) -- (-9, 1)
    (-8, 1) -- (-7, 1)    (-6, 1) -- (-4, 1)
    (5, 1) -- (9, 1)    (12, 1) -- (13, 1)
    (-14, 0) -- (-12, 0)    (-11, 0) -- (-10, 0)
    (-9, 0) -- (-6, 0)    (4, 0) -- (6, 0)
    (9, 0) -- (12, 0)
    (-15, -1) -- (-14, -1)    (-12, -1) -- (-10, -1)
    (-9, -1) -- (-7, -1)    (4, -1) -- (8, -1)
    (9, -1) -- (12, -1)    (13, -1) -- (14, -1)
    (-14, -2) -- (-13, -2)    (-12, -2) -- (-11, -2)
    (-10, -2) -- (-8, -2)    (6, -2) -- (7, -2)
    (9, -2) -- (10, -2)    (13, -2) -- (15, -2)
    (-15, -3) -- (-14, -3)    (-13, -3) -- (-12, -3)
    (-11, -3) -- (-9, -3)    (-7, -3) -- (-5, -3)
    (-4, -3) -- (-3, -3)    (4, -3) -- (5, -3)
    (6, -3) -- (8, -3)    (9, -3) -- (13, -3)
    (14, -3) -- (15, -3)
    (-13, -4) -- (-12, -4)    (-11, -4) -- (-8, -4)
    (-7, -4) -- (-4, -4)    (-3, -4) -- (-2, -4)
    (5, -4) -- (6, -4)    (9, -4) -- (10, -4)
    (11, -4) -- (12, -4)    (13, -4) -- (14, -4)
    (-15, -5) -- (-13, -5)    (-12, -5) -- (-11, -5)
    (-10, -5) -- (-9, -5)    (3, -5) -- (7, -5)
    (8, -5) -- (13, -5)    (14, -5) -- (15, -5)
    (-14, -6) -- (-13, -6)    (-11, -6) -- (-6, -6)
    (-1, -6) -- (0, -6)    (1, -6) -- (6, -6)
    (7, -6) -- (11, -6)    (12, -6) -- (15, -6)
    (-13, -7) -- (-8, -7)    (-5, -7) -- (-3, -7)
    (-1, -7) -- (1, -7)    (2, -7) -- (3, -7)
    (6, -7) -- (7, -7)    (9, -7) -- (10, -7)
    (13, -7) -- (15, -7)
    (-15, -8) -- (-14, -8)    (-13, -8) -- (-12, -8)
    (-11, -8) -- (-9, -8)    (-8, -8) -- (-7, -8)
    (-4, -8) -- (-3, -8)    (-2, -8) -- (0, -8)
    (1, -8) -- (2, -8)    (4, -8) -- (5, -8)
    (6, -8) -- (8, -8)    (9, -8) -- (10, -8)
    (11, -8) -- (15, -8)
    (-14, -9) -- (-12, -9)    (-2, -9) -- (2, -9)
    (3, -9) -- (4, -9)    (8, -9) -- (14, -9)
    (-15, -10) -- (-14, -10)    (-13, -10) -- (-12, -10)
    (-7, -10) -- (-6, -10)    (-1, -10) -- (0, -10)
    (1, -10) -- (3, -10)    (4, -10) -- (5, -10)
    (7, -10) -- (9, -10)    (10, -10) -- (14, -10)
    (-15, -11) -- (-12, -11)    (-8, -11) -- (-5, -11)
    (-3, -11) -- (-1, -11)    (0, -11) -- (3, -11)
    (4, -11) -- (5, -11)    (7, -11) -- (11, -11)
    (12, -11) -- (15, -11)
    (-15, -12) -- (-13, -12)    (-10, -12) -- (-9, -12)
    (-8, -12) -- (-5, -12)    (-3, -12) -- (-2, -12)
    (-1, -12) -- (0, -12)    (3, -12) -- (7, -12)
    (9, -12) -- (11, -12)    (12, -12) -- (15, -12)
    (-15, -13) -- (-14, -13)    (-11, -13) -- (-10, -13)
    (-9, -13) -- (-7, -13)    (-6, -13) -- (-5, -13)
    (-4, -13) -- (-3, -13)    (1, -13) -- (4, -13)
    (5, -13) -- (7, -13)    (8, -13) -- (10, -13)
    (12, -13) -- (14, -13)
    (-15, -14) -- (-14, -14)    (-12, -14) -- (-11, -14)
    (-10, -14) -- (-9, -14)    (-4, -14) -- (2, -14)
    (3, -14) -- (5, -14)    (6, -14) -- (9, -14)
    (10, -14) -- (11, -14)    (13, -14) -- (15, -14)
    (-15, -15) -- (-12, -15)    (-11, -15) -- (-4, -15)
    (-3, -15) -- (-2, -15)    (-1, -15) -- (2, -15)
    (4, -15) -- (6, -15)    (7, -15) -- (8, -15)
    (11, -15) -- (14, -15)
    ;
    \draw[s-hi]
    (0,0) -- (0,3) -- (-1,3) -- (-1,9) -- (-3,9)
    (-2,9) -- (-2,10) -- (-3,10) -- (-3,13) -- (-4,13)
        -- (-4, 14) -- (-5,14)
    (-3,12) -- (-4,12) -- (-4, 13)
    (-1,7) -- (-2,7) -- (-2,6)
    (-1,6) -- (-4,6)
    (-3,6) -- (-3,8) -- (-7,8)
    (-3,7) -- (-4,7) -- (-4,8)
    (-4,7) -- (-6,7) -- (-6,8)
    (0,0) -- (-4,0) -- (-4,-1) -- (-5,-1)
        -- (-5,-2)  -- (-7,-2)
    (-3,0) -- (-3,2)
    (-2,0) -- (-2,-1)
    (-1,0) -- (-1,2) -- (-2,2) -- (-2,3)
    (-1,1) -- (1,1)
    (0,0) -- (2,0) -- (2,1) -- (3,1)
    (1,0) -- (1,2) -- (0,2) -- (2,2) -- (2,3) -- (0,3) -- (5,3)
    (2,2) -- (2,8) -- (7,8)
    (4,3) -- (4,8) (2,4) -- (4,4)
    (4,6) -- (7,6) -- (7,5) -- (8,5) -- (8,7) -- (9,7)
        -- (9,6) -- (10,6) -- (10,7)
    (4,4) -- (5,4) -- (5,5) -- (9,5) -- (9,6) -- (12,6)
    (5,6) -- (5,10)
    (6,6) -- (6,7) -- (5,7) -- (5,9) -- (6,9)
    (-1,0) -- (-1,-2) -- (-3,-2)
    (-2,-2) -- (-2,-3)
    (-1,-1) -- (1,-1) -- (1,0)
    (0,-2) -- (3,-2)
    (0,-3) -- (2,-3) -- (2,-2)
    (0,0) -- (0,-5) -- (-5,-5) -- (-5,-9)
    (-1,-5) -- (-1,-7) -- (-2,-7) -- (-2,-10) -- (-4,-10)
        -- (-4,-11) -- (-3,-11) -- (-3,-10)
    (-2,-5) -- (-2,-6) -- (-3,-6) -- (-3,-5)
    (-3,-6) -- (-5,-6) -- (-5,-7) -- (-6,-7)
    (-5,-8) -- (-6,-8) -- (-6,-9) -- (-9,-9)
    (-8,-9) -- (-8,-11)
    ;
    \foreach \x / \y in
        {
        -9/-9,-8/-9,-8/-10,-8/-11,-7/8,-7/-2,-7/-9,-6/7,-6/8,-6/-2,-6/-7,
        -6/-8,-6/-9,-5/7,-5/8,-5/14,-5/-1,-5/-2,-5/-5,-5/-6,-5/-7,-5/-8,
        -5/-9,-4/0,-4/6,-4/7,-4/8,-4/12,-4/13,-4/14,-4/-1,-4/-5,-4/-6,
        -4/-10,-4/-11,-3/0,-3/1,-3/2,-3/6,-3/7,-3/8,-3/9,-3/10,-3/11,
        -3/12,-3/13,-3/-2,-3/-5,-3/-6,-3/-10,-3/-11,-2/0,-2/2,-2/3,-2/6,
        -2/7,-2/9,-2/10,-2/-1,-2/-2,-2/-3,-2/-5,-2/-6,-2/-7,-2/-8,-2/-9,
        -2/-10,-1/0,-1/1,-1/2,-1/3,-1/4,-1/5,-1/6,-1/7,-1/8,-1/9,
        -1/-1,-1/-2,-1/-5,-1/-6,-1/-7,0/0,0/1,0/2,0/3,0/-1,0/-2,0/-3,
        0/-4,0/-5,1/0,1/1,1/2,1/2,1/3,1/-1,1/-2,1/-3,
        2/0,2/1,2/2,2/3,2/4,2/5,2/6,2/7,2/8,2/-2,2/-3,
        3/1,3/3,3/4,3/8,3/-2,4/3,4/4,4/5,4/6,4/7,4/8,
        5/3,5/4,5/5,5/6,5/7,5/8,5/9,5/10,
        6/5,6/6,6/7,6/8,6/9,7/5,7/6,7/8,8/5,8/6,8/7,
        9/5,9/6,9/7,10/6,10/7,11/6,12/6
        }
        \filldraw[xshift=\x cm, yshift=\y cm, s-hi, fill=white, line width=1.0pt] (0,0) circle (4pt);
    \filldraw[black, fill=white, line width=1.0pt]
        (0,0) circle (4pt) node[anchor=south west] {$v$};
    \draw[densely dashed, very thick, s-hi] (-14.5,-14.5) rectangle (14.5, 14.5);
\end{tikzpicture}
\caption{Polluted configuration with $p = 0.56$. The box $B_{d}$ is highlighted, as well as the vertices accessible via unbarred paths from $v$.}
\label{fig:polluted}
\end{figure}
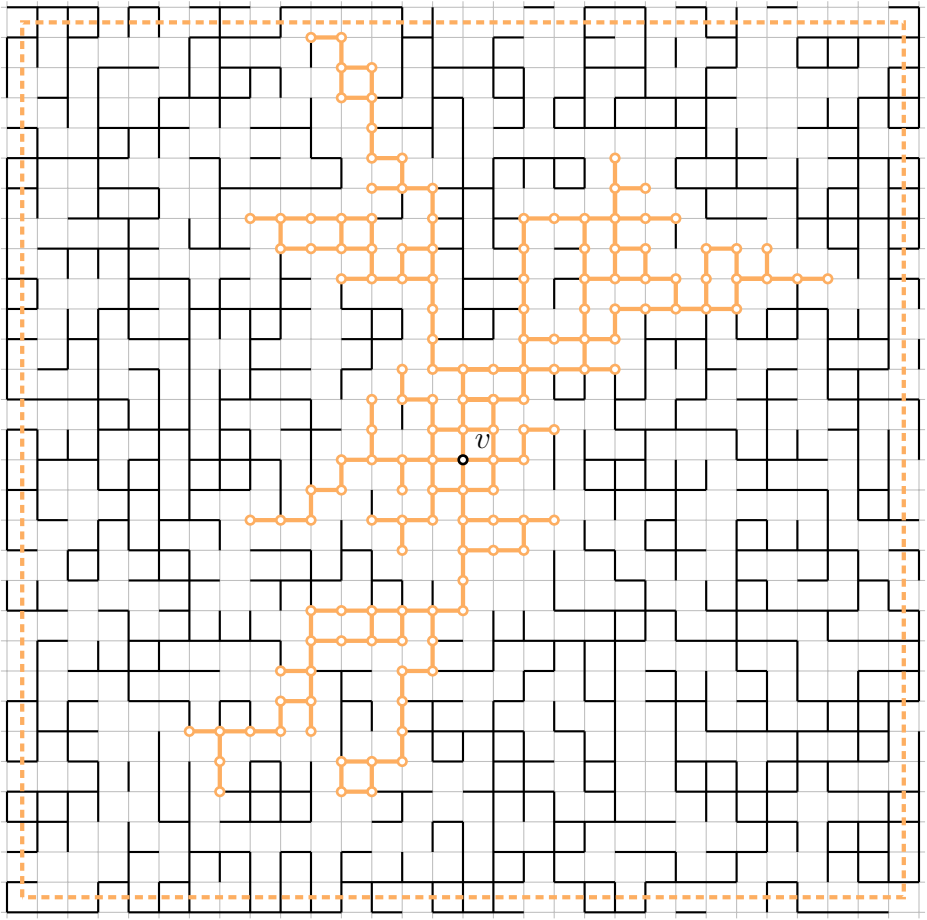

Finally, we show that a.s. \Cref{str:barred} is a winning strategy for Breaker.

\begin{proof}[Proof of \Cref{prop:barred-strategy}]
Consider a barred configuration of edges of $\ZZ^2$ and a vertex $v$ chosen by Maker.
Translate the board so that $v = (0,0)$, and note that the configuration is still barred.
We show that Breaker wins the $(1,1)$-game on this board by following \Cref{str:barred}.

If the origin is in a finite open cluster, Breaker wins immediately, so assume it is in an infinite open cluster.
By \Cref{lem:inframsey}, there is $d$ such that every unbarred path starting from the origin is contained within $B_d$.

Since Breaker always claims an unclaimed edge inside the box $B_{d+1}$ if it is possible, all of $B_{d+1}$ is claimed after a finite number of rounds.
We claim that once every edge inside $B_{d+1}$ is claimed, there is no path from $(0,0)$ to $B_{d+2} \setminus B_{d+1}$ consisting solely of open edges claimed by Maker.
This clearly implies Breaker's win.

Assume for contradiction there was such a Maker's path $P$.
It cannot be an unbarred path, since every unbarred path is fully contained in $B_d$.

Denote $P = \set{v_1, \dotsc, v_\ell}$, where $v_0 = (0,0)$ and $v_\ell \in B_{d+2} \setminus B_{d+1}$.
Let $k$ be maximal such that subpath $\set{v_0, \dotsc, v_{k-1}, v_{k}}$ is unbarred.
This implies that $2 \leq k \leq \ell - 2$, and further that the edges $\set{v_{k-1},v_k}$ and $\set{v_k,v_{k+1}}$ are paired and contained in $B_{d+1}$.
Therefore, by \Cref{str:barred} and the barrier pairing in \Cref{def:pairing}, Maker could not have claimed both these edges.
Therefore, no such path $P$ exists.
\end{proof}

Thus, the proof of \Cref{prop:barred-strategy} and hence also of \Cref{thm:polluted} is finished.


\section{Open problems}
\label{sec:open}

In this paper we have made progress on some of the questions that Day and Falgas-Ravry~\cite{day2020maker2} asked.
Notably, we have improved the upper bound on the value of the ratio parameter $\rho^*$ (if such $\rho^*$ indeed exists).

However, some of their other questions still remain open, as well as some new problems, which we present below.
There are several directions for further research that we consider to be of a great interest.

\begin{description}[labelindent=\parindent, leftmargin=2\parindent]

\item[Maker's win]
We still do not know any integers $(m,b)$ with $b>1$ and $m<2b$ for which Maker has a winning strategy in the $(m,b)$-game on $\ZZ^2$.
In fact, we do not know any such integers even in the boosted version of the game, where Maker is allowed to claim arbitrary but finite number of edges before her first turn.
Maker's strategies in the papers of Day and Falgas-Ravry~\cites{day2020maker1,day2020maker2} suggest that it may be beneficial for Maker to play as a dual Breaker in some auxiliary game.
Hence, perhaps some of the techniques developed in the present paper could help in answering those questions.

\item[Breaking the perimetric ratio with boosted Maker]
In \Cref{sec:ratio}, we have shown that there exists $\delta>0$ such that provided $m$ is large enough and $b \geq (2-\delta)m$, Breaker has a winning strategy in the $(m,b)$-game on $\ZZ^2$.
It would be interesting to study the boosted variant of this game, where Maker is allowed to claim finitely many edges before her first turn.
There, we still do not know of any integers $(m,b)$ with $b<2m$ for which Breaker has a winning strategy in the boosted $(m,b)$-game (though as we have shown in \Cref{sec:fast-boost}, $b=2m$ suffices even in this variant).
Again, it is possible that some of the techniques of the present paper can be developed further to handle this case as well.

\item[Polluted board]
In \Cref{sec:polluted} we considered the case when the $(1,1)$-game is played on a random board obtained by running a usual percolation process with parameter $p$ on $\ZZ^2$ (with origin chosen by Maker).
This board has almost surely an infinite connected component whenever $p>1/2$, yet we have seen that when $p<0.6298$, Breaker almost surely has a winning strategy.
We still do not know if there exists any $\varepsilon > 0$ such that whenever $p > 1-\varepsilon$, Maker almost surely has a winning strategy.
Moreover, note that such a strategy would have to be very different from Maker's winning strategy in the $(1,1)$-game on $\ZZ^2$ which was suggested in~\cite{day2020maker2}.
It might also be of an interest to investigate the general $(m,b)$-game on a polluted board.
Even more generally, it could be interesting to consider the polluted game played on other boards $\Lambda$, and see if similar phenomena occur there.
\end{description}

One could also try to improve the bounds in \Cref{thm:main1} and \Cref{thm:polluted} derived in this paper.
But while improving these constants significantly would be of some interest (note that improving \Cref{thm:main1} by little should not be very hard as we did not try to optimise the constant there fully), we believe that breaking the three barriers highlighted above is even more important.

For a further overview of more open questions, see the paper of Day and Falgas-Ravry~\cite{day2020maker2}.
Note that many seemingly easy questions are still open, for instance, it is not known who wins the $(2,3)$ or $(3,2)$-game on $\ZZ^2$.

\section*{Acknowledgements}
\label{sec:ack}

The authors would like to thank their PhD supervisor Professor Béla Bollobás for suggesting this problem and for his comments.




\end{document}